\documentclass[11pt]{article}
\usepackage{amsmath,amsthm}
\usepackage{amssymb}
\usepackage{amsfonts}
\usepackage{mathrsfs}
\usepackage{shortvrb,psfrag}
\usepackage{enumerate}
\usepackage{color}
\usepackage[dvips]{graphicx}

\numberwithin{equation}{section}

\def\ba{\begin{eqnarray}}
\def\ea{\end{eqnarray}}

\def\R{\Bbb R}

\def\Z{\Bbb Z}

\def\no{\noindent}

\def\endproof{\hphantom{MM}\hfill\llap{$\square$}\goodbreak}

\newcommand{\lp}{\left(}
\newcommand{\rp}{\right)}

\DeclareMathOperator{\trace}{Tr}
\DeclareMathOperator{\divergence}{div}
\DeclareMathOperator{\Sym}{Sym}

\newcommand{\beq}{\begin{equation}}
\newcommand{\eeq}{\end{equation}}
\newcommand{\ben}{\begin{eqnarray}}
\newcommand{\een}{\end{eqnarray}}
\newcommand{\beno}{\begin{eqnarray*}}
\newcommand{\eeno}{\end{eqnarray*}}

\newtheorem{Theorem}{Theorem}[section]

\newtheorem{Proposition}[Theorem]{Proposition}
\newtheorem{Lemma}[Theorem]{Lemma}

\newtheorem{Remark}[Theorem]{Remark}

\makeatletter
\def\blfootnote{\xdef\@thefnmark{}\@footnotetext}
\makeatother

\begin{document}
\title{Further results on the fractional Yamabe problem:\\the umbilic case }

\author{Mar\'ia del Mar Gonz\'alez$^\ddag$, Meng Wang$^\dag$
\\[2mm]
{\small $ ^\dag$  Department of Mathematics, Zhejiang University, Hangzhou
310027, P. R. China}\\
{\small E-mail:  mathdreamcn@zju.edu.cn}\\[2mm]
{\small $ ^\ddag$ Univ. Polit\`ecnica de Catalunya, Spain}\\
{\small E-mail: mar.gonzalez@upc.edu}\\[2mm]
}

\blfootnote{M.d.M. Gonz\'alez is supported by Spain Government project MTM2011-27739-C04-01 and GenCat 2009SGR345. M. Wang is supported by NSFC 11371316. Both authors would like to acknowledge the hospitality of Princeton University where this work was initiated.}

\date{\today}
\maketitle

\begin{abstract}
 We prove some existence results  for the fractional Yamabe problem in the case that the boundary manifold is umbilic, thus covering some of the cases not considered by Gonz\'alez and Qing.  These are inspired by the work of Coda-Marques on the boundary Yamabe problem but, in addition, a careful understanding of the behavior at infinity for asymptotically hyperbolic metrics is required.
\end{abstract}


\section{Introduction and statement of results}

Suppose that $X^{n+1}$ is a smooth manifold with smooth boundary $M^{n}$ for $n\ge 3$.  A function $\rho$ is a defining function on the boundary $M^n$ in $X^{n+1}$ if
\begin{equation*}
\rho>0~\mbox{in}~X^{n+1},~\rho=0~\mbox{on}~M^n,~d\rho\neq0~\mbox{on}~M^{n}.
\end{equation*}
We say that a Riemannian metric $g^{+}$ on $X^{n+1}$ is conformally compact if, for some defining function $\rho$, the metric $\bar{g}=\rho^2g^{+}$ extends smoothly to $\overline{X}^{n+1}$. This induces a conformal class of metrics $\hat{h}=\bar{g}|_{TM^n}$ on $M^n$ as defining functions vary. The conformal manifold $(M^n, [\hat{h}])$ is called the conformal infinity of $(X^{n+1},g^{+})$.

A metric $g^{+}$ is said to be asymptotically hyperbolic if it is conformally compact and the sectional curvature approaches $-1$ at infinity, which is equivalent to $|d\rho|_{\bar{g}}=1$ on $M^{n}$. If we have that $Ric[g^{+}]=-ng^{+}$, then we call $(X^{n+1},g^{+})$ a conformally compact Einstein manifold. In these settings, given a representative $\hat h$ of the conformal infinity, there exists a unique defining function $\rho$ such that in a tubular neighborhood near $M$, the metric $g^+$ has the normal form
\begin{equation}\label{normal}g^+=\frac{d\rho^2+h_\rho}{\rho ^2},\end{equation}
where $h_\rho$ is a one-parameter family of metrics on $M$ satisfying $h_0=\hat h$. In the Einstein case we may assume that $h_\rho$ as an asymptotic expansion which is even in powers of $\rho$. This is only true up to order $n$, but it will not be relevant to our study (see \cite{graham} for an introduction). We also denote $\bar g=\rho^2 g^+$.

For the rest of the paper, we will fix $\gamma\in(0,1)$. The conformal fractional Laplacian $P_\gamma^{\hat h}$ is constructed as the Dirichlet-to-Neumann operator for the scattering problem for $(X,g^+)$. In particular, from \cite{mazzeomelrose} and \cite{grahamzworski}, it is known that if  given $f\in \mathcal C^{\infty}(M)$, then for all but a discrete set of values $s\in\mathbb C$,  the generalized eigenvalue problem
\ba\label{generalizedeigen1}
-\Delta_{g^{+}}u-s(n-s)u=0,~~\mbox{in}~X,
\ea
has a solution of the form
\ba\label{generalizedeigen2}
u=F\rho^{n-s}+G\rho^s,~~F,G\in \mathcal C^{\infty}(\bar{X}),~~F|_{\rho=0}=f.
\ea
The scattering operator on $M$ is defined as
$$S(s)f=G|_{M},$$
and it is a meromorphic family of pseudo-differential operators in whole complex plane. In fact, the values $s=\frac{n}{2}$, $\frac{n}{2}+1$,... are simple poles of finite rank, these are called the trivial poles. $S(s)$ may have other poles (corresponding to the $L^2$-eigenvalues for $-\Delta_{g^+}$), but we will assume in the rest of the paper that we are not in such cases. More precisely, we will require that $\lambda_1(-\Delta_{g^+})>\frac{n^2}{4}-\gamma^2$, if one writes $s=\frac{n}{2}+\gamma$ for $\gamma\in(0,1)$ (this condition on $-\Delta_{g^+}$ was not written in \cite{gonzalezchang} but it should be added in \cite{gonzalezqing} for the study of the fractional Yamabe problem). Then the conformal fractional Laplacian on $(M,\hat h)$ is defined as
\begin{equation*}
P_{\gamma}^{\hat h}=d_{\gamma}S(\frac{n}{2}+\gamma),\quad \text{for a constant }\, d_{\gamma}=2^{2\gamma}\frac{\Gamma(\gamma)}{\Gamma(-\gamma)}.
\end{equation*}
Here the dependence on $g^+$ is always implicitly understood. With this normalization, the principal symbol of the operator $P_\gamma^{\hat h}$ equals that of $(-\Delta_{\hat{h}})^{\gamma}$.
The operators $P_{\gamma}^{\hat{h}}$ satisfy the following conformal covariance property: under a conformal change of metric
$$\hat{h}_{w}=w^{\frac{4}{n-2\gamma}}\hat{h}, \quad w>0,$$
we have
\ba\label{p}
P_{\gamma}^{\hat{h}_{w}}\phi=w^{-\frac{n+2\gamma}{n-2\gamma}}P_{\gamma}^{\hat{h}}(w\phi),
\ea
for all smooth functions $\phi$. One can also define ``fractional order curvature"
\ba\label{q}
Q_{\gamma}^{\hat g}:=P_{\gamma}^{\hat h}(1).
\ea
From \eqref{p} and \eqref{q}, we obtain the fractional curvature equation
\ba\label{qequation}
P_{\gamma}^{\hat{h}}w=w^{\frac{n+2\gamma}{n-2\gamma}}Q_{\gamma}^{\hat{h}_{w}}
\ea

The fractional Yamabe problem for $\gamma\in(0,1)$ was introduced in \cite{gonzalezqing}. In that paper the authors consider the following scale-free functional on metrics in the conformal class $[\hat{h}]$ given by
\begin{equation*}
I_{\gamma}[\hat{h}]=
\frac{\int_{M}Q_{\gamma}^{\hat{h}}\,dv_{\hat{h}}}{(\int_{M}\,dv_{\hat{h}})^{\frac{n-2\gamma}{n}}}.
\end{equation*}
$I_{\gamma}$ called the $\gamma-$Yamabe functional. Once $\hat{h}$ is fixed, one can write
\begin{equation*}
I_{\gamma}[w,\hat{h}]:=I_{\gamma}[\hat{h}_{w}]=\frac{\int_{M}w P_{\gamma}^{\hat{h}}w\, dv_{\hat{h}}}{\left(\int_{M}w^{2^{*}}\,dv_{\hat{h}}\right)^{\frac{2}{2^{*}}}},
\end{equation*}
where $2^{*}=\frac{2n}{n-2\gamma}$. The corresponding {\it $\gamma-$Yamabe problem} is   to find a metric in the conformal class $[\hat{h}]$ that minimizes the $\gamma-$Yamabe functional $I_{\gamma}$. As in the scalar curvature case one defines the $\gamma-$Yamabe  constant by
\begin{equation*}
\Lambda_{\gamma}(M,[\hat{h}])=\inf\{I_{\gamma}[h]:h\in [\hat{h}]\}.
\end{equation*}
It is clear that $\Lambda_{\gamma}(M,[\hat{h}])$ is an invariant in the conformal class $[\hat{h}]$ when $g^{+}$ is fixed.

In particular, if $w$ is a minimizer of $I_{\gamma}[w,\hat{h}]$, then the metric $\hat{h}_{w}$ has constant fractional  curvature; indeed, such $w$ is a solution to
\begin{equation}\label{equation0}
P_{\gamma}^{\hat{h}}w=c w^{\frac{n+2\gamma}{n-2\gamma}}.
\end{equation}
It is well known (\cite{gonzalezqing}) that the sign of such constant $c$ is equal (or zero) to the one of $\Lambda(M,[\hat h])$.

The non-local equation \eqref{qequation} on $M$ may be written as a degenerate elliptic problem in $X$. Indeed, one has the following extension problem (see  \cite{gonzalezchang,gonzalezqing,case-chang}). For the rest of the paper, we consider $\gamma\in(0,1)$ and we write $a=1-2\gamma$.
\begin{Lemma}[\cite{gonzalezchang}]\label{lemma1}
Let $(X,g^+)$ be an asymptotically hyperbolic manifold as explained above. Given $f\in \mathcal C^{\infty}(M)$, the generalized eigenvalue problem \eqref{generalizedeigen1}-\eqref{generalizedeigen2} is equivalent to
\begin{equation*}
\left\{
\begin{split}
-\divergence(\rho^a\nabla U)+E(\rho)U=0~~\mbox{in}~~(X,\bar{g}),\\
U|_{\rho=0}=f~~\mbox{on}~M,
\end{split}
\right.
\end{equation*}
where $U=\rho^{s-n}u$ and $U$ is the unique minimizer of the energy
$$F[V]=\int_{X}\rho^a|\nabla V|_{\bar{g}}^2\,dv_{\bar{g}}
+\int_{X} E(\rho)|V|^2\,dv_{\bar{g}}$$
among all the functions $V\in W^{1,2}(X,\rho^a)$ with fixed trace $V|_{\rho=0}=f$.
Here
$$
E(\rho)=\rho^{-1-s}(-\Delta_{g^{+}}-s(n-s))\rho^{n-s}.
$$
or equivalently,
\begin{equation*}\label{curvature}
E({\rho})=\frac{n-1+a}{4n}\left[R[\bar{g}]-\left(n(n+1)+R[g^{+}]\right)\rho^{-2}\right]\rho^a.
\end{equation*}
Moreover,
\begin{itemize}
\item[1.] For $\gamma\in(0,\frac{1}{2})$,
\ba\label{equality}
P_{\gamma}^{\hat{h}}f=-d_{\gamma}^{*}\lim_{\rho\rightarrow 0}\rho^a\partial_{\rho}U.
\ea
\item[2.] For $\gamma=\frac{1}{2}$,
\begin{equation*}\label{gamma12}
P_{\frac{1}{2}}^{\hat{h}}f=-\lim_{\rho\rightarrow 0}\partial_{\rho}U+\tfrac{n-1}{2}Hf,
\end{equation*}
where $H$ is the mean curvature of $(M,\hat h)$.

\item[3.] For $\gamma\in (\frac{1}{2},1)$, expression \eqref{equality}  holds if and only if $H=0$.
\end{itemize}
Here the constant is given by
\begin{equation}\label{constant-d}
d_{\gamma}^{*}=\frac{2^{2\gamma-1}\Gamma(\gamma)}{\gamma\Gamma(-\gamma)}.\end{equation}
\end{Lemma}

In the following we assume that $H\equiv 0$ in the case $\gamma\in(1/2,1)$. Note that this is automatically true in the Einstein case since the term $h_\rho$ in the normal form \eqref{normal} for the metric $g^+$ only has even terms in the expansion.

We also define the functional
\begin{equation}\label{barI}
\overline I_{\gamma}[U,\hat h]=\frac{d_{\gamma}^{*}\int_{X}\left({\rho}^a|\nabla U|_{\bar{g}}^2+E(\rho)U^2\right)\,dv_{\bar{g}}}
{(\int_{M}U^{2^{*}}\,dv_{\hat{h}})^{2/{2^{*}}}}.
\end{equation}
As a consequence of Lemma \ref{lemma1}, a minimizer $\overline I_\gamma$  will give a minimizer for the $\gamma$-Yamabe functional $I_\gamma$. In particular, if one defines
\begin{equation*}
\overline\Lambda_{\gamma}(X,[\hat{h}])=\inf\{\overline I_{\gamma}[U,\hat h]:U\in W^{1,2}(X,\rho^a)\},
\end{equation*}
then
$$\overline\Lambda_{\gamma}(X,[\hat{h}])=\Lambda_{\gamma}(M,[\hat{h}]).$$

We define the usual fractional Sobolev norm on $M$
$$\|w\|^2_{H^\gamma(M)}:=\|w\|^2_{L^2(M)}+ \int_M
w(-\Delta_{\hat h})^\gamma w\,dv_{\hat h},$$
and the weighted norm in the extension
$$
\|U\|_{W^{1,2}(X,\rho^a)}^2= \int_X \rho^a|\nabla U|_{\bar g}^2
\,dv_{\bar g}+\int_{X} \rho^a U^2  \,dv_{\bar g}.
$$
Thus the minimization problem for the functional \eqref{barI} is related to the well known trace Sobolev embedding
$$W^{1,2}(M,\rho^a)\to H^\gamma(M)\to L^{2^*}(M).$$
(see the papers \cite{gonzalezqing} and \cite{jinxiong}, and the references therein). On the Euclidean case $M=\mathbb R^n$, $X=\mathbb R^{n+1}_+$ the best constant in the Sobolev inequality above may be explicitly calculated. Indeed, for every $U\in W^{1,2}(\mathbb R^{n+1}_+,y^a)$, let $w:=U(\cdot,0)$, then
\begin{equation}\label{Euclidean-Sobolev}\|w\|_{L^{2^*}(\R^n)}^2\leq \bar
S(n,\gamma)\int_{\mathbb R^{n+1}_+} y^a |\nabla U|^2\,dxdy,\end{equation} where
\begin{equation*}\label{S-bar}\bar S(n,\gamma):= d^*_\gamma S(n,\gamma),\quad S(n,\gamma)=\frac{\Gamma\lp\frac{n-2\gamma}{2}\rp}{\Gamma\lp\frac{n+2\gamma}{2}\rp}|vol(\mathbb S^n)|^{-\frac{2\gamma}{n}}.
\end{equation*}
Equality holds if and only if
$$w(x)=c\left(\frac{\mu}{|x-x_0|^2+\mu^2}\right)^{\frac{n-2\gamma}{2}},\quad x\in\R^n,$$
for $c\in\mathbb R$, $\mu>0$ and $x_0\in\R^n$ fixed, and $U$
is the Poisson extension of $w$ given by
$$U(x,y)=\int_{\mathbb R^n} \frac{y^{1-a}}{\lp|x-\xi|^2+y^2\rp^{\frac{n+1-a}{2}}}\,w(\xi)\,d\xi.$$
In addition, \eqref{Euclidean-Sobolev} allows to calculate the best $\gamma$-Yamabe constant on the sphere with its canonical metric as the boundary of the Poincar\'e ball $\Lambda_\gamma(\mathbb S^n,[g_c])$ by stereographic projection. Indeed,
\begin{equation}\label{Sobolev-sphere}\Lambda_\gamma(\mathbb S^n,[g_c])=\frac{1}{S(n,\gamma)}.\end{equation}

The manifold version of \eqref{Euclidean-Sobolev} was considered in \cite{jinxiong}. From their results one can show that, in general,
$$\Lambda_{\gamma}(M,[\hat{h}])>-\infty.$$

We have all the ingredients needed to handle the fractional Yamabe problem. Indeed, as in the standard Yamabe problem (cf. \cite{leeparker, Schoen-Yau:book}), one must compare the value of the Yamabe constant to the one on the sphere:

\begin{Proposition}[\cite{gonzalezqing}]\label{previous-work}
Fix $\gamma\in(0,1)$. Let $(X^{n+1},g^{+})$ be an asymptotically hyperbolic manifold with conformal infinity $(M,[\hat h])$ as explained above and assume, in addition, that $H=0$ when $\gamma\in(\frac{1}{2},1)$. Then,
$$\Lambda_{\gamma}(M,[\hat{h}])\leq\Lambda_{\gamma}(\mathbb S^n,[g_c]).$$
Moreover, the strict inequality
\begin{equation}\label{condition-Lambda}
\Lambda_{\gamma}(M,[\hat{h}])<\Lambda_{\gamma}(\mathbb S^n,[g_c])
\end{equation}
ensures that the $\gamma-$Yamabe problem for $(M, \hat h)$ is solvable.
\end{Proposition}

The question is now when the strict inequality is attained. Note that when $\gamma=1/2$, in the conformally compact Einstein setting, the lower order term has a very simple expression $E(\rho)=\frac{n-1}{4n}R_{\bar{g}}$ and the functional simply reduces to
\begin{equation*}
\overline I_{1/2}[U,\hat h]=\int_{X^{n+1}}\frac{\left(|\nabla U|_{\bar{g}}^2+\tfrac{n-1}{4n}R_{\bar g}U^2\right)\,dv_{\bar{g}}}
{(\int_{M^n}U^{2^{*}}\,dv_{\hat{h}})^{2/{2^{*}}}}.
\end{equation*}
Thus the $1/2$-Yamabe problem is almost exactly the boundary problem proposed by Escobar in \cite{escobar} and later studied by Marques \cite{marques}, Han and Li \cite{Han-Li} and Brendle \cite{Brendle}, for instance.  The problem consists of looking for a conformal metric  on $(X,\bar g)$ of zero scalar curvature and constant mean curvature on the boundary. Escobar \cite{escobar} considered the case that $M$ as a non-umbilic point for dimensions $n>5$, and some other particular cases. Marques completed the umbilic case for large dimensions under some non-vanishing conditions on the Weyl tensor.

For the fractional Yamabe problem, the only case that has been studied so far is when $M$ contains a non-umbilic point under some dimension and curvature restrictions (see \cite{gonzalezqing}). In particular, there it is assumed that
\ba\label{condition}\rho^{-2}\left(R[g^{+}]-Ric[g^{+}](\rho\partial\rho)+n^2\right),
~\mbox{as}~\rho\rightarrow 0.\ea
Their main result is the construction of a suitable test function near the non-umbilic point satisfying
\begin{equation*}
\Lambda_{\gamma}(M,[\hat{h}])<\Lambda_{\gamma}(\mathbb S^n,[g_c])
\end{equation*}
and hence the $\gamma-$Yamabe problem is solvable for $\gamma\in (0,1)$. Note that condition \eqref{condition} is an intrinsic curvature condition of an asymptotically hyperbolic manifold, which is independent of the choice of geodesic defining functions.

On the other hand, compactness and asymptotic behavior results for Palais-Smale sequences for fractional Laplacian equations with critical nonlinearities such as \eqref{equation0} were considered in  \cite{Palatucci-Pisante:1,Palatucci-Pisante:2,Fang-Gonzalez}.\\

The main purpose of this paper is to use Marques results in \cite{marques} on the umbilic case in order to give further results on the solvability of the fractional Yamabe problem for any $\gamma\in(0,1)$. In the proof we need to use the construction of conformally compact Einstein metrics with prescribed conformal infinity by Fefferman and Graham \cite{Fefferman-Graham}.

 \begin{Theorem}\label{theorem1}
 Fix $n\ge 5$. Suppose that $({X}^{n+1},g^{+})$ is an $(n+1)$-dimensional asymptotically hyperbolic manifold with conformal infinity $(M,[\hat h])$ satisfying
 \begin{align}
&\hat h^{ij} F_{ij}|_{\rho=0}=0, \label{F-trace}\\
&\partial_\rho F|_{\rho=0}=0, \label{F-derivative}\\
&\hat h^{ij}\partial_{\rho\rho\rho} F_{ij}|_{\rho=0}=0,\label{F-third}
\end{align}
where $F$ is the tensor
$$F[g^+]=\rho(Ric[g^+]+ng^+).$$
Assume that $M$ is umbilic. Then if there is a point $p\in M$ such  that
\begin{equation}\label{hypothesis1}Ric_{\rho\rho,\rho}[\bar g](p)<0,\end{equation}
then
\begin{equation*}
\Lambda_{\gamma}(M,[\hat{h}])<\Lambda_{\gamma}(\mathbb S^n,[g_c]),
\end{equation*}
\no
where ${Ric}[\bar g]$ is the Ricci tensor for the metric $\bar{g}=\rho^2g^{+}$
and $\rho$ is the geodesic defining function that appears in the normal form of $g^+$  with respect to the choice of conformal representative $\hat h$ in the  conformal infinity.
 \end{Theorem}

\begin{Remark} As shown in Lemma \ref{lemmah3} and \eqref{2.22}, the condition of the  existence  of  point $p$ satisfying \eqref{hypothesis1} is intrinsic for $g^{+}$, and it does not depend on the choice of the representative in the conformal class $(M,[\hat h])$. Note also that condition \eqref{F-trace} is precisely \eqref{condition}. Moreover, the umbilicity condition together with \eqref{F-trace} imply that $F|_{\rho=0}=0$. In the Einstein case, $F\equiv 0$.
\end{Remark}

\begin{Theorem}\label{theorem2}
Fix $n>5+2\gamma$. Suppose that $({X}^{n+1},g^{+})$ is an $(n+1)$-dimensional conformal compact Einstein manifold with conformal infinity $(M,[\hat h])$ and such that $M$ is  umbilic. Then if there is a point $q\in M$ such that  ${W}[\hat h](q)\neq 0$,
\begin{equation*}
\Lambda_{\gamma}(M,[\hat{h}])<\Lambda_{\gamma}(\mathbb S^n,[g_c]).
\end{equation*}
Here $W[\hat h]$ stands for the Weyl tensor of the metric $\hat h$.
\end{Theorem}

\begin{Remark}
The condition on the Weyl tensor in the theorem above is also conformal invariant on $M$. As we will see in the proof of this theorem, it is enough to assume that the first and third terms $h^{(1)}$ and $h^{(3)}$ in the expansion of the metric $h_\rho$ vanish, which is weaker than the Einstein condition.
\end{Remark}

The idea of the proof of both theorems is to find a suitable test function to calculate the value of the functional \eqref{barI} and compare it to its value on the sphere. The first step is to choose a  particular background metric $(X,\bar g)$ with very precise asymptotic behavior near $p$. However, in contrast to the works of Escobar \cite{escobar} and Marques \cite{marques} on the $1/2$-Yamabe problem, where they are free to choose conformal Fermi coordinates on the whole extension manifold $(X,\bar g)$, our freedom of choice of metrics is restricted to the boundary. Once $h_1\in [\hat h]$ is chosen, then the metric $\bar g_1$ is uniquely given by the defining function $\rho_1$ appearing in the normal form \eqref{normal}, i.e., $\bar g_1=(\rho_1)^2\bar g$. Hence we will make some assumptions on the behavior of the asymptotically hyperbolic manifold in order to have a suitable background metric on the conformal infinity, and we will develop some generalized conformal Fermi coordinates.

 \section{Suitable conformal Fermi coordinates}

We fix $(X^{n+1},\bar g)$ a smooth Riemannian manifold with boundary $M^n$, and let $\hat h=\bar g|_M$. As we have mentioned in the introduction, we need to choose a very particular background metric for $X$ near an umbilic point $p\in M$.

We follow the notation from \cite{marques}. Throughout this section we will make use of the index notation for tensors; commas will denote covariant differentiation. When dealing with manifolds with boundary, we will use the indices $1\le i,j,k,l,m,p,r,s\le n$ and $1\le a,b,c,d\le n+1$. The Greek letters $\alpha$ and $\beta$ will be multiindices. In Fermi coordinates on a neighborhood $M\times [0,\epsilon)$ the letter $t$ will refer to the normal direction to $M$, and we can write
$$\bar g=\hat h+h^{(1)}t+h^{(2)}t^2+h^{(3)}t^3+h^{(4)} t^4+o(t^4).$$
In particular, $h^{(1)}$ is the second fundamental form on $M$ (up to a constant factor), and the mean curvature (up to a constant factor) is given by
\begin{equation*}\label{mean-curvature}H=\frac{1}{n}\trace_{\hat h} h^{(1)}.\end{equation*}
We say that a point $p\in M$ is umbilic if the tensor $T_{ij}=h_{ij}^{(1)}-H\bar{g}_{ij}$ vanishes at $p$.

We will denote $\nabla$ the covariant derivative and by $R_{abcd}$ the full Riemannian curvature tensor. The Ricci tensor will be denoted by $Ric_{ab}$, the scalar curvature by $R$. The Weyl tensor will be denoted by $W$. Tensors in the metric $\bar g$ will be over-lined; an object without lines will be given with respect to the boundary metric $\hat h$. We will also use the definition
$$  \Sym_{i_1...i_r}T_{i_1...i_r}=\frac{1}{r!}\sum_{\sigma}T_{i_{\sigma(1)}...i_{\sigma(r)}},$$
where $\sigma$ ranges over all the permutations of the set $\{1,...,r\}$.

We finally recall that
\begin{equation}\label{geometry1}\trace_{\hat h} h^{(2)}=-2Ric[\bar g](\partial_\rho)+\tfrac{1}{2}\|h^{(1)}\|^2.\end{equation}
and
\begin{equation}\label{geometry2}
R[\bar{g}]
=2Ric[\bar{g}](\partial_{\rho})+R[\hat{h}]+\tfrac{1}{4}\left(\|h^{(1)}\|_{\hat{h}}^{2}-H^2\right).
\end{equation}

The following  lemma is about expansions for the metric  $\bar{g}_{ij}$ under an additional hypothesis on the second fundamental form at $p\in \partial X$,

\begin{Lemma}[\cite{marques}]\label{inverseg}
Suppose ${\nabla}_{\alpha}h_{ij}^{(1)}=0$ at $p\in \partial M$ for every $|\alpha|\le 3$. Then, in Fermi coordinates around $p$,
\begin{equation}\label{expansioninverseg}\begin{split}
\bar{g}^{ij}(x,t)&=\delta_{ij}+\frac{1}{3}{R}_{ikjl}x_kx_l+\bar{R}_{titj}t^2\\
&+\frac{1}{6}{R}_{ikjl,m}x_kx_lx_m+\bar{R}_{titj,k}t^2x_k+\frac{1}{3}\bar{R}_{titj,t}t^3\\
&+\left(\frac{1}{20}{R}_{ikjl,mp}+\frac{1}{15}{R}_{iksl}{R}_{jmsp}\right)x_kx_lx_mx_p\\
&+\left(\frac{1}{2}\bar{R}_{titj,kl}+\frac{1}{3}\Sym_{ij}({R}_{iksl}\bar{R}_{tstj})\right)t^2x_kx_l\\
&+\frac{1}{3}\bar{R}_{titj,tk}t^3x_k\\
&+\frac{1}{12}(\bar{R}_{titj,tt}+8\bar{R}_{tits}\bar{R}_{tstj})t^4+O(r^5),
\end{split}\end{equation}
where $r=|(x,t)|$, and the curvatures are evaluated at $p$.
In addition,
\begin{equation}\label{determinant1}\begin{split}
\det\bar{g}&=1-\tfrac{1}{6}Ric_{kl,m}x_k x_l x_m -\overline{Ric}_{tt,k}t^2 x_k\\
&-\tfrac{1}{3}\overline{Ric}_{tt,t}t^3+\left(-\tfrac{1}{20}Ric_{kl,mp}
-\tfrac{1}{90}R_{iksl}R_{imsp}\right)x_kx_lx_mx_p\\
& -\tfrac{1}{2}\overline{Ric}_{tt,kl} t^2x_k x_l-\frac{1}{3}\overline{Ric}_{tt,t k}t^3x_k\\
& +\tfrac{1}{24}(-2\overline{Ric}_{tt,tt}-4(\overline{R}_{t it j})^2)t^4+O(|(x,t)|^5).
\end{split}\end{equation}

\end{Lemma}

Next, as we have mentioned in the introduction, given an asymptotically hyperbolic manifold $(X^{n+1},g^{+})$ and a representative $\hat{h}$ of the conformal infinity $(M^{n},[\hat{h}])$, one can find a geodesic defining function $\rho$ such that in a neighborhood $M\times(0,\epsilon)$ of $X$ the metric $g^+$
has the form
\begin{equation}\label{normal-form}
 g^{+}=\frac{d{\rho}^{2}+h_{\rho}}{\rho^2},
 \end{equation}
where $h_\rho$ is a 1-parameter family of metrics on $M$ satisfying $h_0=\hat h$. We say that such $g^+$ is written in normal form.
We write
\ba\label{compacted}
\bar{g}=\rho^2g^{+}=d\rho^2+h_{\rho}=d\rho^2+\hat{h}+h^{(1)}\rho+h^{(2)}\rho^2
+h^{(3)}\rho^3+h^{(4)}\rho^4+o(\rho^4)
\ea
near the conformal infinity. One may define an umbilic point $p\in M$ for the asymptotically hyperbolic case if such point is umbilic with respect to this metric $\bar g$.
If every point at the boundary is umbilic, we say that the asymptotically hyperbolic manifold has umbilical boundary.

Note that the set of umbilic points of the boundary is a conformal invariant. Assume that we are given  $\rho$ and $\tilde \rho$ two different geodesic defining functions of $M$ in $X$ associated with
representatives $\hat h$ and $\tilde h$ of the conformal infinity
$(M^n, \ [\hat h])$, respectively. We may write
$$
g^+ = \rho^{-2}(d\rho^2 + h_\rho) = \tilde\rho^{-2}(d\tilde\rho^2 +
\tilde h_{\tilde\rho})
$$
near $M$,
where
$$
h_\rho = \hat h + \rho h^{(1)} + O(\rho^2),\quad
\tilde h_{\tilde\rho} = \tilde h + \tilde\rho \tilde h^{(1)} +
O(\tilde\rho^2)
$$
near the conformal infinity. Then it was proven in \cite{gonzalezqing} that
$$
\tilde h^{(1)} = \left.(\tilde{\rho}/{\rho})\right|_{\rho=0} h^{(1)} \text{\quad on $M$}.
$$
In particular
$$
H  = \left.(\tilde\rho/\rho)\right|_{\rho = 0} \tilde H  \text{\quad on
$M$}.
$$

 In the following lemmas we will present some technical results on the expansion on the metric written in normal form \eqref{normal-form}-\eqref{compacted} near the conformal infinity under some extra geometric assumptions. These will be needed in the proof of the main proposition in this section.

\begin{Lemma}\label{lemma-normal-form} Suppose that $(X^{n+1}, g^+)$
is an asymptotically hyperbolic manifold and $\rho$ is a geodesic
defining function associated with a representative $\hat h$ of the
conformal infinity $(M^n, [\hat h])$ such that $g^+$ is written in normal form.
Assume that $X$ has umbilical boundary, and that \eqref{F-trace}-\eqref{F-third} hold.
Then, for every point on the boundary $\rho=0$,
\begin{align}
&\label{mean-cero} H =0,\quad h^{(1)}=0,\\
&\label{h2} \trace_{\hat h}h^{(2)} = \frac{R[\hat h]}{2(1-n)}\\
\label{h22}
&h^{(2)}=\frac{R[\hat{h}]\hat{h}+2(1-n)Ric[\hat{h}]}{2(n-2)(n-1)},\\
&\trace_{\hat{h}}h^{(4)}
=\frac{R[h_{\rho}]_{,\rho\rho}|_{\rho=0}-2(n-2)\|h^{(2)}\|_{\hat{h}}^{2}}{8(2-n)}.
\label{h4-trace}
\end{align}
\end{Lemma}

\begin{proof} The ideas come from \cite{gonzalezqing} and go back to the work of Fefferman and Graham \cite{Fefferman-Graham} on the construction of Einstein metrics with prescribed conformal infinity. Recall formula (2.5) from \cite{graham}
\begin{equation}\label{formula-graham}
 \rho h_{ij}'' + (1-n)h_{ij}'  - h^{kl}h_{kl}'h_{ij} - \rho h^{kl}h_{ik}'h_{jl}' +
\frac 12 \rho h^{kl}h_{kl}' h_{ij}'- 2\rho Ric_{ij}[h_\rho]  =
F_{ij}, \end{equation}
where $h_{ij}$ denotes the tensor $h:=h_{\rho}$, derivation $'$ denotes $\partial_{\rho}$, and $Ric[h_{\rho}]$ denotes the Ricci tensor of $h_{\rho}$ with $\rho$ fixed.

In the first step, taking trace in \eqref{formula-graham} with respect to $h_\rho$ gives
\begin{equation*} \rho\trace_h h''  + (1-2n)\trace_{h}h'- \rho\|h'\|_h^2 +
\tfrac 1 2 \rho (\trace_h h')^2- 2\rho R[\hat h] = \trace_{h_\rho} F,\end{equation*}
which implies, using \eqref{F-trace}, that
$$\trace_{\hat h} h^{(1)}=0\quad \text{at }\rho=0.$$
Together with the umbilicity condition we can conclude \eqref{mean-cero} and, as a particular consequence, $F|_{\rho=0}= 0$.

Next, we differentiate \eqref{formula-graham} with respect to $\rho$ and set $\rho=0$. We obtain
\begin{equation}\label{first-derivative}
(2-n)h''_{ij}+\|h'\|^2 h_{ij}-(\trace_{\hat h}{h''})h_{ij}-\tfrac{1}{2}(\trace_{\hat h}{h'})h'_{ij}
-h^{kl}h'_{ik}h'_{jl}-2R_{ij}=\partial_\rho F_{ij} \mbox{ at }\rho=0.
\end{equation}
Taking the trace, and using  \eqref{mean-cero} and that we are umbilic we arrive at
\begin{equation*}
\trace_{\hat{h}}h''=\frac{R[\hat{h}]}{1-n} \mbox{ at }\rho=0,
\end{equation*}
which immediately yields \eqref{h2}.
As a consequence, we also have from \eqref{first-derivative}, recalling that we are in the umbilic boundary case and \eqref{F-derivative}, that
\begin{equation*}
h^{(2)}=\frac{R[\hat{h}]\hat{h}+2(1-n)Ric[\hat{h}]}{2(n-2)(n-1)}.
\end{equation*}

Differentiating \eqref{formula-graham} three times and setting $\rho=0$ (again, recalling that we are umbilic so all the terms with $h'_{ij}$ drop out) gives
\begin{equation}\label{h4}
24(4-n)h^{(4)}_{ij}+12\|h^{(2)}\|_{\hat{h}}^{2}\hat{h}_{ij}-24\trace_{\hat{h}}h^{(4)}\hat{h}
-24\hat{h}^{kl}h_{ik}^{(2)}h_{jl}^{(2)}=6Ric_{ij}[h_{\rho}]_{,\rho\rho}|_{\rho=0}
+\partial_{\rho\rho\rho} F_{ij}|_{\rho=0},
\end{equation}
where we note that
\begin{equation*}\begin{split}
h_{\rho}^{kl}&=\hat{h}^{kl}-\hat{h}^{kr}h_{rs}^{(2)}\hat{h}^{sl}\rho^2\\
&+(\hat{h}^{ks}h_{sp}^{(2)}\hat{h}^{pp'}h_{p'r}^{(2)}\hat{h}^{rl}
-\hat{h}^{ks}h_{sr}^{(4)}\hat{h}^{rl})\rho^4+O(\rho^6).
\end{split}\end{equation*}
Take trace to \eqref{h4} gives
\begin{equation*}\label{traceh4}
24(4-2n)\trace_{\hat{h}}h^{(4)}
+(12n-24)\|h^{(2)}\|_{\hat{h}}^{2}=6\hat{h}^{ij}Ric_{ij}[h_{\rho}]_{,\rho\rho}|_{\rho=0}
+\trace_{\hat h}\partial_{\rho\rho\rho} F|_{\rho=0},
\end{equation*}
from where we obtain \eqref{h4-trace}, recalling \eqref{F-third}. This completes the proof of the proposition.\\
\end{proof}

The following lemma, together with \eqref{2.22}, shows that condition \eqref{hypothesis1} is independent of the choice of representative in the conformal infinity $(M,[\hat h])$.

\begin{Lemma}\label{lemmah3}
Let  $(X^{n+1},g^{+})$ be an asymptotically hyperbolic manifold with umbilical boundary. Let $\hat h$ be another representative of the conformal class $[\hat h]$, and let $\rho$ and $\tilde \rho$ be the geodesic defining functions associated with $\hat{h}$ and $\tilde h$, respectively, such that $g^+$ is written in normal form in both cases. Assume that conditions \eqref{F-trace}-\eqref{F-third} are is satisfied. Then at $\rho=0$,
$$\trace_{\hat{h}}h^{(3)}=\trace_{\tilde{h}}\tilde{h}^{(3)}e^{3w}.$$
\end{Lemma}

\begin{proof}
We follow \cite{graham} on the construction of the normal form for both $\hat h$ and $\tilde h$. Let $\tilde{\rho}=e^{w}\rho$ near the conformal infinity, then
$$1=|d(e^{w}\rho)|_{e^{2w}\rho^2g^{+}}^{2},$$
which implies
\begin{equation}\label{formula70}2\partial_\rho w+\rho\left[(\partial w_\rho)^2+|\nabla w|_{h_{\rho}}^2\right]=0,\end{equation}
and
\begin{equation}\label{formula71}\partial_{\rho\rho\rho}w=-\partial_{\rho}(|\nabla w|^2_{h_{\rho}}),\end{equation}
on $\rho=0$.

Next, since we write
$$ g^{+}=\frac{d\rho^2+\hat{h}+h^{(2)}\rho^2+h^{(3)}\rho^3+O(\rho^4)}{\rho^2}
 =\frac{d\tilde{\rho}^2+\tilde{h}+\tilde{h}^{(2)}\tilde{\rho}^2+\tilde{h}^{(3)}\tilde{\rho}^3
 +O(\tilde{\rho}^4)}{\tilde{\rho}^2},$$
comparing the coefficients of $\rho^3$, we must have that
$$h^{(3)}+h^{(1)}w_{\rho\rho}
+\frac{1}{3}\hat{h}w_{\rho\rho\rho}=\frac{1}{2}\tilde{h}^{(1)}e^{-w}w_{\rho\rho}+\tilde{h}^{(3)}e^{w},$$
and thus
$$\trace_{\hat{h}}h^{(3)}+\frac{n}{3}w_{\rho\rho\rho}=\trace_{\tilde{h}}\tilde{h}^{(3)}e^{3w}.$$
We claim that $w_{\rho\rho\rho}=0$ on $\rho=0$. This is so because from \eqref{formula71} we can write
$$w_{\rho\rho\rho}=-\partial_\rho \lp h^{ij}\partial_i w\partial_j w\rp=(\partial_\rho h^{ij}) \partial_i w\partial_j w+2 h^{ij}\partial_{i\rho}w\partial_j w.$$
The first term vanishes at $\rho=0$ since $h^{(1)}\equiv 0$ on $M$, while the second vanishes too because  \eqref{formula70} implies that $w_\rho\equiv 0$ on $M$.

 The proof of the lemma is completed.
\end{proof}

\medskip
Now we are ready for the main result in this section: the construction of conformal Fermi coordinates in the asymptotic hyperbolic case.

\begin{Proposition}\label{metric1}
Suppose that $(X^{n+1},g^{+})$ is asymptotically hyperbolic manifold with umbilical boundary and \eqref{F-trace}-\eqref{F-third} hold. Then given a point $p\in M$, there exists  a representative $\hat{h}$ of the conformal infinity such that, for the metric written in normal form, we have: $g^+=\rho^{-1}(d\rho^2+h_\rho)=\rho^{-2}\bar g$,
\begin{itemize}
\item[(i).] $H=0$ on $M$,
\item[(ii).] $Ric[\hat{h}](p)=0$ on $M$,
\item[(iii).] $Ric[\bar{g}](\partial_{\rho})(p)=0 $ on $M$,
\item[(iv).] $R[\bar{g}](p)=0$ on $M$,
\item[(v).] The expansion for the determinant of the metric, assuming $p$ to be the origin of the coordinate system $\{x^1,\ldots,x^n\}$ on $M$,
\begin{equation}\label{detg}\begin{split}
\det\bar{g}&=1-\tfrac{1}{3}\overline{Ric}_{\rho\rho,\rho}[\bar g](p)\rho^3\\
&+\left\{-\tfrac{1}{20}Ric_{kl,mp}[\hat h]
-\tfrac{1}{90}R_{iksl}[\hat h]R_{imsp}[\hat h]\right\}(p)\,x_kx_lx_mx_p\\
& -\tfrac{1}{2}{Ric}_{\rho\rho,kl}[\bar g](p) \rho^2x_k x_l-\frac{1}{3}{Ric}_{\rho\rho,\rho k}[\bar g](p)\rho^3x_k\\
& +\tfrac{1}{24}\left\{-2{Ric}_{\rho\rho,\rho\rho}[\bar g]-4({R}_{\rho i\rho j}[\bar g])^2\right\}(p)\,\rho^4+O(|(x,\rho)|^5).
\end{split}\end{equation}
\item[(vi).] $\Sym(R_{ij,kl}[\hat{h}]+\frac{2}{9}R_{pijm}[\hat{h}]R_{pklm}[\hat{h}])(p)=0$.
\item[(vii).] And for the derivatives of the Ricci curvature,
\begin{align}
\label{2.22} &{Ric}_{\rho\rho,\rho}[\bar g](p)=-3\trace_{\hat{h}}h^{(3)}(p),\\
&{Ric}_{\rho\rho,kk}[\bar g](p)=\frac{R_{,ii}[\hat h](p)}{2(n-1)}=-\frac{|W|^2[\hat h](p)}{12(n-1)}.\label{2.16}
\end{align}
Here $W[\hat h]$ is the Weyl tensor for the metric $\hat h$.
\end{itemize}
Moreover, if $(X,g^{+})$ is a conformally compact Einstein manifold, written in normal form as
\begin{equation*}
g^{+}=\frac{d\rho^2+h_{\rho}}{\rho^2}
=\frac{d\rho^2+\hat{h}+h^{(2)}\rho^2+h^{(4)}\rho^4+O(\rho^6)}{\rho^2},
\end{equation*}
we also have
\begin{align}\label{2.15}
&{Ric}_{\rho\rho,\rho}[\bar g](p)=0,\\
&{R}_{\rho i\rho j}[\bar g](p)=0,\label{2.17}\\
&{Ric}_{\rho\rho,\rho\rho}[\bar g](p)=0,\label{2.18}\\
&{R}_{\rho i\rho j,ij}[\bar g](p)=\frac{ R_{,ii}[\hat h](p)}{2(n-1)}=-\frac{|W|^2[\hat h](p)}{12(n-1)},\label{2.19}\\
&{R}_{,\rho\rho}[\bar g](p)=0.\label{formula84}
\end{align}
\end{Proposition}

\begin{proof}
We fix $p\in M$. The proof uses \cite[Theorem 5.1]{leeparker} on the existence of conformal normal coordinates $\{x^1,\ldots,x^n\}$ on $M$ centered at $p$. In particular, we can choose a representative $\hat{h}$ of the conformal infinity  such
 that, at $p$:
\begin{itemize}
\item[(a)] $Ric_{ij}[\hat{h}]=0$,
\item[(b)] $Ric_{ij,k}[\hat{h}]+Ric_{jk,i}[\hat{h}]+Ric_{ki,j}[\hat{h}]=0,$
\item[(c)] $\Sym(Ric_{ij,kl}[\hat{h}]+\frac{2}{9}R_{pijm}[\hat{h}]R_{pklm}[\hat{h}])=0,$
\item[(e)] $R_{,ii}[\hat h]=-\frac{1}{6}|W|^2[\hat h]$, and moreover, near $p$, $R[\hat h]=O(|x|^2)$.
\end{itemize}
We immediately get that properties $(ii)$ and $(vi)$ are true.

We know that there exists a geodesic defining function $\rho$ such that we can write $g^+$ in normal form \eqref{normal-form}-\eqref{compacted}. From Lemma \ref{lemma-normal-form} and the umbilicity property we must have
$$h^{(1)}=0\quad\text{and}\quad H=0\quad \text{on }M,$$
which in particular implies that
\begin{equation*}\label{gradients}\nabla_{\alpha}H=0\quad \text{and}  \quad\nabla_{\alpha}h^{(1)}=0\quad\text{on }M.\end{equation*}

In the following, we will use overline for curvatures referring to $\bar g$, while without overline will mean quantities with respect to the metric $\hat h$.

Next, we look at the metric $\bar g$ near $p$. Statement \emph{(iii)} follows from \eqref{geometry1} and \eqref{h2}, using \emph{(ii)}, while \emph{(iv)} is an immediate consequence of \eqref{geometry2}.

Now we look at the expansion for the determinant $\det(\bar g)$ in the umbilic case given in \eqref{determinant1}. The term with $x_{k}x_lx_m$ vanishes because of the choice of $\hat h$ satisfying condition (b). In addition, recalling \eqref{geometry1} and \eqref{h2}, we have that in
in the umbilic case
$$\overline{Ric}_{\rho\rho,k}=-\tfrac{1}{2}\left(\trace_{\hat h}h^{(2)}\right)_{,k}=\tfrac{1}{4(1-n)}R_{,k}.$$
Using (e) above we see that $\overline{Ric}_{\rho\rho,k}$ vanishes at the point. Thus from \eqref{determinant1} and the previous remarks we obtain \emph{(v)}.

Finally, we show \emph{(vii)}. For an expansion
$$h_{\rho}=\hat{h}+h^{(2)}\rho^2+h^{(3)}\rho^3+h^{(4)}\rho^4+O(\rho^5),$$
we have that
\begin{equation}\label{formula10}
\det h_{\rho}=\det\hat{h}\left[1+\trace_{\hat{h}}h^{(2)}\rho^2+\trace_{\hat{h}}h^{(3)}\rho^3
+\left\{\trace_{\hat{h}}h^{(4)}
+\tfrac{1}{2}(\trace_{\hat{h}}h^{(2)})^2-\tfrac{1}{2}\|h^{(2)}\|_{\hat{h}}^{2}\right\}\rho^4
+O(\rho^5)\right].
\end{equation}
We first recall formula \eqref{h2} for $\trace_{\hat{h}}h^{(2)}$. Comparing the coefficients of $\rho^2$ in \eqref{formula10} with \eqref{detg} we must have that, at the point $p$,
$$
\overline{Ric}_{\rho\rho,kk}=\frac{R_{,ii}}{2(n-1)},
$$
and \eqref{2.16} follows from property (e) above. Next, comparing the coefficients of $\rho^3$ we obtain that at the point $p$,
$$-\tfrac{1}{3}\overline{Ric}_{\rho\rho,\rho}=\trace_{\hat{h}}h^{(3)},$$
which shows \eqref{2.22}.\\

From now on we assume that $h^{(3)}$ and $h^{(5)}$ vanish. In this case, we have the asymptotics
$$h_{\rho}=\hat{h}+h^{(2)}\rho^2+h^{(4)}\rho^4+O(\rho^6),$$
In particular, \eqref{formula10} reduces to
\ba\label{det}
\det(h_{\rho})=\det\hat{h}\left[1+\trace_{\hat{h}}{h^{(2)}}\rho^2
+\left\{ \trace_{\hat{h}}h^{(4)}+\tfrac{1}{2}(\trace_{\hat{h}}h^{(2)})^2
-\tfrac{1}{2}\|h^{(2)}\|_{\hat{h}}^{2}\right\}\rho^4+O(\rho^6)\right].
\ea
Comparing the coefficients of in the asymptotics of \eqref{detg} and \eqref{det} we conclude that, recalling (a), (e),  \eqref{h2}, \eqref{h22} and \eqref{h4-trace}, at a point $p\in M$,
\begin{align}\label{2.27}
&\overline{Ric}_{\rho\rho,\rho}=0,\\
\label{2.28}
&-\frac{1}{2}\overline{R}_{\rho\rho,kk}=\frac{R_{,ii}}{4(1-n)},\\
&-\frac{1}{12}\overline{Ric}_{\rho\rho,\rho\rho}-\frac{1}{6}(\overline{R}_{\rho i\rho j})^2=\frac{\hat{h}^{ij}Ric_{ij}[h_{\rho}]_{,\rho\rho}|_{\rho=0}}{8(2-n)}
=\frac{R''[h_\rho]|_{\rho=0}}{8(2-n)}.\label{formula50}
\end{align}
Equation \eqref{2.27} is precisely \eqref{2.15}. On the other hand, note that
\begin{equation*}\begin{split}
Ric_{ij}[h_{\rho}]&=-{h}^{kl}_\rho\left[\frac{1}{2}\left(\frac{\partial^2h_{kl}(\rho)}{\partial x^i\partial x^j}+\frac{\partial^2h_{ij}(\rho)}{\partial x^k\partial x^l}-\frac{\partial^2h_{il}(\rho)}{\partial x^k\partial x^j}-\frac{\partial^2h_{kj}(\rho)}{\partial x^i\partial x^l}\right)\right.\\
&+\left.\Gamma_{kl}^{p}(h_\rho)\Gamma_{ij}^{r}(h_\rho)h_{pr}(\rho)-\Gamma_{il}^{p}(h_\rho)
\Gamma_{kj}^{r}(h_\rho)h_{pr}(\rho)\right].
\end{split}\end{equation*}
By \eqref{h22}, taking derivatives twice in $\rho$, the expression above simplifies to
\begin{equation}\label{formula52}\begin{split}
\hat{h}^{ij}Ric_{ij}''[h(\rho)](p)&
=-\hat{h}^{ij}(p)\hat{h}^{kl}(p)\left(\frac{\partial^2h^{(2)}_{kl}}{\partial x^i\partial x^j}+\frac{\partial^2h^{(2)}_{ij}}{\partial x^k\partial x^l}-\frac{\partial^2h^{(2)}_{il}}{\partial x^k\partial x^j}-\frac{\partial^2h^{(2)}_{kj}}{\partial x^i\partial x^l}\right)(p)\\
&=-2\left(\frac{\partial^2h_{kk}^{(2)}}{\partial (x^i)^2}-\frac{\partial^2h^{(2)}_{ki}}{\partial x^k\partial x^i}\right)(p).
\end{split}\end{equation}
Using the formula for $h^{(2)}$ from \eqref{h22}
\begin{equation}\label{formula51}\begin{split}
\frac{\partial^2h_{kk}^{(2)}}{\partial{(x^{i})}^2}&-\frac{\partial^2h_{ki}^{(2)}}{\partial x^k\partial x^i}\\
=&\frac{1}{2(n-2)(n-1)}\left([nR_{,ii}+2(1-n)Ric_{kk,ii}-2(1-n)Ric_{ki,ki}-R_{,ii}\right].
\end{split}\end{equation}
But, contracting the Bianchi identity
$$R_{lkjm,ii}+R_{lkij,mi}+R_{lkmi,ji}=0,$$
on the    indices $l,j$ and again $k,m$, we get
$$Ric_{ki,ki}=\tfrac{1}{2}R_{,ii},$$
so we get that expression \eqref{formula51} vanishes at $p$.
Thus, from \eqref{formula50} and \eqref{formula52} we can conclude that
\ba\label{2.32}
-\frac{1}{12}\overline{Ric}_{\rho\rho,\rho\rho}-\frac{1}{6}(\bar{R}_{\rho i\rho j})^2=0.
\ea
On the other hand, for every point on $M$,
\begin{equation}\label{2.33}\begin{split}
\bar{R}_{\rho i\rho j}&=-\frac{1}{2}\left(\frac{\partial^2\bar{g}_{\rho\rho}}{\partial x^i\partial x^j}
+\frac{\partial^{2}\bar{g}_{ij}}{\partial \rho^2}
-\frac{\partial^2\bar{g}_{i\rho}}{\partial \rho\partial x^j}
-\frac{\partial^2\bar{g}_{\rho j}}{\partial \rho\partial x^j}\right)\\
&\quad-\Gamma_{\rho\rho}^{q}\Gamma_{ij}^{s}\,\bar{g}_{qs}+\Gamma_{i\rho}^{q}\Gamma_{\rho j}^{s}\,\bar{g}_{qs}\\
&=-h^{(2)}_{ij},
\end{split}\end{equation}
which vanishes at the point $p$. Thus \eqref{2.17} holds. Moreover, putting together \eqref{2.32} with \eqref{2.33} we arrive to conclusion \eqref{2.18}.

Moreover, differentiating \eqref{2.33} on the tangential variables, recalling \eqref{h22},
$$
\bar{R}_{\rho i\rho j,ij}=-\frac{R_{,ij}\hat{h}_{ij}+2(1-n)Ric_{ij,ij}}{2(n-2)(n-1)},
$$
which, after evaluating at $p$ yields
\begin{equation}\label{formula80}
\bar{R}_{\rho i\rho j,ij}=\tfrac{1}{2(n-1)}R_{,ii}.
\end{equation}
This shows \eqref{2.19}. In addition, we recall the second Bianchi identity
\begin{equation*}\label{second-Bianchi}\bar{R}_{abcd,\rho}+\bar{R}_{ab\rho c,d}+\bar{R}_{abd\rho,c}=0.\end{equation*}
In particular, contracting
$$\bar{R}_{abcd,\rho\rho}+\bar{R}_{ab\rho c,d\rho}+\bar{R}_{abd\rho,c\rho}=0,$$
gives
\begin{equation}\label{formula81}\bar{R}_{,\rho\rho}
=2\overline{Ric}_{i\rho,i\rho}+2\overline{Ric}_{\rho\rho,\rho\rho}\end{equation}
and contracting
$$\bar{R}_{abcd,\rho i}+\bar{R}_{ab\rho c,di}+\bar{R}_{abd\rho,ci}=0$$
yields
\begin{equation}\label{formula82}
\overline{Ric}_{\rho i,\rho i}=\overline{Ric}_{\rho\rho,kk}-\bar R_{\rho i\rho j,ij}.
\end{equation}
Thus from \eqref{formula80} and \eqref{2.28} we conclude that
\begin{equation}\label{formula83}\overline{Ric}_{\rho i,\rho i}=0\end{equation}
and the point $p$.
Next, from \eqref{formula81}, interchanging the order of covariant differentiation and recalling \eqref{2.18},
\begin{equation*}
\bar{R}_{,\rho\rho}=2\overline{Ric}_{\rho i,\rho i}+2(\bar{R}_{a\rho i\rho}\overline{Ric}_{ai}+\bar{R}_{aii\rho}\overline{Ric}_{\rho a})=2\overline{Ric}_{\rho i,\rho i},
\end{equation*}
where we have used \emph{(iii)} and \eqref{2.17} to cancel out terms. As a consequence, from \eqref{formula83}, we conclude
$$\bar{R}_{,\rho\rho}=0,$$
which is \eqref{formula84}.
\end{proof}

\section{Some technical lemmas in $\mathbb R^{n+1}_+$}

We only consider the case $\gamma\in (0,1)\setminus \{1/2\}$, since $\gamma=1/2$ is much simpler. For the rest of the section, we also assume that $n>4+2\gamma$.

At first, we review the following fact about Bessel functions (see section 9.6.1 in \cite{abramowitz})
\begin{Lemma}\label{bessel0}
The solution of ODE
\ba\label{ode}
\partial_{ss}\phi+\frac{a}{s}\partial_{s}\phi-\phi=0
\ea
maybe written as $\phi(s)=s^{\gamma}\psi(s)$, for $a=1-2\gamma$, where $\psi$ solves that is well known Bessel
 equation
 \ba\label{bessel}
 s^{2}\psi^{''}+s\psi^{'}-(s^2+\gamma^2)\psi=0.
 \ea
 In addition, \eqref{bessel} has two independent solutions, $I_{\gamma}$, $K_{\gamma}$, which are the modified Bessel functions. Their asymptotic behavior is given precisely by
 \begin{eqnarray*}
 I_{\gamma}(s)\sim \frac{1}{\Gamma(\gamma+1)}\left(\frac{s}{2}\right)^{\gamma}
 \left(1+\frac{s^2}{4(\gamma+1)}+\frac{s^4}{32(\gamma+1)(\gamma+2)}+...\right),
 \end{eqnarray*}
 \begin{eqnarray*}
 K_{\gamma}(s)&\sim& \frac{\Gamma(\gamma)}{2}\left(\frac{2}{s}\right)^{\gamma}\left(1+\frac{s^2}{4(\gamma+1)}+\frac{s^4}{32(\gamma+1)(\gamma+2)}+...\right)\\
 && +\frac{\Gamma(-\gamma)}{2}\left(\frac{s}{2}\right)^{\gamma}\left(1+\frac{s^2}{4(\gamma+1)}+\frac{s^{4}}{32(\gamma+1)(\gamma+2)}+...\right),
 \end{eqnarray*}
 for $s\rightarrow 0^{+}$, $\gamma\notin \Z$. And when $s\rightarrow +\infty$,
 \begin{eqnarray*}
 I_{\gamma}(s)\sim \frac{1}{\sqrt{2\pi s}}e^{s}\left(1-\frac{4\gamma^2-1}{8s}+\frac{(4\gamma^2-1)(4\gamma^2-9)}{2!(8s)^2}-...\right),\\
 K_{\gamma}(s)\sim \sqrt{\frac{\pi}{2s}}e^{-s}\left(1+\frac{4\gamma^2-1}{8s}+\frac{(4\gamma^2-1)(4\gamma^2-9)}{2!(8s)^2}+...\right).
 \end{eqnarray*}
\end{Lemma}

\medskip
We have the following identities:

\begin{Lemma}\label{bessel1} Let  $\phi(s)=s^{\gamma}K_{\gamma}(s)$ be the solution to \eqref{ode} (up to multiplicative constant).  Then:
\begin{align}
&\int_{0}^{+\infty}s^{a+3}\left(\phi^2+{\phi'}^2\right)ds=\frac{3(a+2)}{2}\int s^{a+1}\phi^2ds,\label{10}\\
&\int_{0}^{+\infty}s^{a+2}\phi'^2ds=\frac{3+a}{3-a}\int_{0}^{\infty}s^{a+2}\phi^2ds,\label{11}
\\
&\int_{0}^{+\infty}s^{a+4}{\phi'}^2ds=\frac{(a+5)(a+3)}{5}\int_{0}^{+\infty}s^{a+2}\phi^2ds,\label{7}\\
&\int_{0}^{+\infty}s^{a+4}\phi^2ds=\frac{(a+3)(5-a)}{5}\int_{0}^{+\infty}s^{a+2}\phi^2ds,\label{8}\\
&\int_{0}^{+\infty}s^{n-4+a}\phi^2ds=\frac{(n-4)(n-5+a)(n-3-a)}{4(n-3)}
\int_{0}^{+\infty}s^{n-6+a}\phi^2ds.\label{9}
\end{align}
\end{Lemma}

 \begin{proof}
 We only prove \eqref{7}, \eqref{8} and \eqref{9} here. Multiply \eqref{ode} by $s^{a+5}\phi'$ and integrate by parts, we get
\begin{equation*}\begin{split}
0&=-\int_{0}^{+\infty}s^{a+5}\phi\cdot\phi'ds+a\int_{0}^{+\infty}s^{a+4}{\phi'}^2ds
+\int_{0}^{+\infty}s^{a+5}\phi'\phi''ds\\
&=\frac{a+5}{2}\int_{0}^{+\infty}s^{a+4}\phi^2ds+a\int_{0}^{+\infty}s^{a+4}{\phi'}^2ds
-\frac{a+5}{2}\int_{0}^{+\infty}s^{a+4}{\phi'}^2ds.
\end{split}\end{equation*}
Then
\ba\label{1st}
\int_{0}^{+\infty}s^{a+4}{\phi'}^2ds=\frac{5+a}{5-a}\int_{0}^{+\infty}s^{a+4}\phi^2ds.
\ea
Next, multiply \eqref{ode} by $s^{a+4}\phi(s)$ and integrate by parts. Using  \eqref{1st}  we can get
\begin{equation}\label{2nd}\begin{split}
0&=-\int_{0}^{+\infty}s^{a+4}\phi^2ds
+a\int_{0}^{+\infty}s^{a+3}\phi\cdot\phi'ds+\int_{0}^{+\infty}s^{a+4}\phi\cdot\phi''ds\\
&=-4\int_{0}^{+\infty}s^{a+3}\phi\cdot\phi'ds
-\int_{0}^{+\infty}s^{a+4}\phi^{2}ds-\int_{0}^{+\infty}s^{a+4}{\phi'}^2ds\\
&=2(a+3)\int_{0}^{+\infty}s^{a+2}\phi^2ds-\frac{10}{5-a}\int_{0}^{+\infty}s^{a+4}\phi^2ds.
\end{split}\end{equation}
Then \eqref{1st} and \eqref{2nd} tell us that
$$\int_{0}^{+\infty}s^{a+4}\phi^2ds=\frac{(a+3)(5-a)}{5}\int_{0}^{+\infty}s^{a+2}\phi^2ds,$$
and
$$\int_{0}^{+\infty}s^{a+4}{\phi'}^2ds=\frac{(a+5)(a+3)}{5}\int_{0}^{+\infty}s^{a+2}\phi^2ds,$$
so \eqref{7} and \eqref{8} are proved.
Next, multiplying \eqref{ode} by $s^{n-4+a}\phi$ and integrating, we can get
\ba\label{5}
-\int_{0}^{+\infty}s^{n-4+a}{\phi'}^2ds
+\frac{(n-4)(n-5+a)}{2}\int_{0}^{+\infty}s^{n-6+a}\phi^2ds=\int_{0}^{+\infty}s^{n-4+a}\phi^2ds.
\ea
On the other hand, multiply \eqref{ode} by $s^{n-3+a}\phi'$ and integrate it; we obtain
\ba\label{6}
\int_{0}^{+\infty}s^{n-4+a}{{\phi}'}^2ds=\frac{n-3+a}{n-3-a}\int_{0}^{+\infty}s^{n-4+a}\phi^2ds.
\ea
Finally, \eqref{5} and \eqref{6} show that \eqref{9} is true.
\end{proof}

\bigskip
The following lemma is very classical:

\begin{Lemma}[\cite{caffarellisilvestre}] \label{lemma-CS}
Given $w\in H^\gamma(\mathbb R^{n})$, there exists a unique solution $U\in W^{1,2}(\mathbb R^{n+1}_+,y^a)$ for the problem
\begin{equation}\label{equation}
\left\{
\begin{array}{l}
{\rm{div}}(y^{a}\nabla U)=0,~~\mbox{in}~\R_{+}^{n+1},\\
U(x,0)=w,~~~\mbox{on}~\R^{n}\times\{0\}.
\end{array}
\right.
\end{equation}
In Fourier variables it is written as
\begin{equation}\label{hatU}\hat{U}(\zeta,y)=\hat{w}(\zeta)\phi(|\zeta|y),\end{equation}
where
\ba\label{phi}\phi(s)=c_1s^{\gamma}K_{\gamma}(s)\ea
for a constant $c_1=\frac{2^{1-\gamma}}{\Gamma({\gamma})}$.
In particular,
\begin{equation}
\label{Poisson}U(x,y)=\mathcal K_\gamma \ast_x w=C_{n,\gamma}\int_{\R^n}\frac{y^{1-a}}{(|x-\tilde x|^2+|y|^2)^{\frac{n+1-a}{2}}}\,w(\tilde x)\,d\tilde x,\end{equation}
where $\mathcal K_\gamma$ is the Poisson kernel for the problem \eqref{equation}.
In addition,
\begin{equation*}
(-\Delta_{\mathbb R^n})^\gamma w=-d_\gamma^*\lim_{y\to 0}y^a\partial_y U,
\end{equation*}
where the constant $d_\gamma^*$ is given by \eqref{constant-d}.
\end{Lemma}

\begin{proof}
We recall some details of the proof for convenience of the reader. Taking Fourier transform  in \eqref{equation} with respect to the variable $x$ we obtain
\begin{equation*}
\left\{
\begin{split}
&-|\zeta|^2\hat{U}(\zeta,y)+\frac{a}{y}\,\hat{U}_{y}(\zeta,y)+\hat{U}_{yy}(\zeta,y)=0,\\
&\hat{U}(\zeta,0)=\hat{w}(\zeta).
\end{split}
\right.
\end{equation*}
Thus we can write
$$\hat{U}(\zeta,y)=\hat{w}(\zeta)\phi(|\zeta|y),$$
where $\phi(t)$ solves the ODE
\begin{equation*}
\left\{
\begin{split}
&\partial_{ss}\phi+\frac{a}{s}\partial_{s}\phi-\phi=0,\quad s\in\mathbb R_+,\\
&\phi(0)=1,~\lim_{s\rightarrow +\infty}\phi(s)=0,
\end{split}
\right.
\end{equation*}
Then Lemma \ref{ode} gives the desired identity \eqref{phi}.
\end{proof}

\bigskip

We denote $|\nabla U|^2=(\partial_{x_1}U)^2+...+(\partial_{x_n}U)^2+(\partial_{y}U)^2$, and $|\nabla_{x}U|^2=(\partial_{x_1}U)^2+...+(\partial_{x_n}U)^2$.

\begin{Lemma}\label{technical}
Let
\begin{equation}\label{diffeo-1}
w(x)=\left(\frac{1}{|x|^2+1}\right)^{\frac{n-2\gamma}{2}},\quad x\in\mathbb R^n,
\end{equation}
 and set
 $U=\mathcal K_{\gamma}\ast_{x}w$ as given in \eqref{Poisson}. Define
\begin{equation*}\begin{split}
&I_1=\int_{\R^{n+1}_{+}}y^{a+2}x_1^2(\partial_{1}U)^2dxdy,~\\
&I_2=\int_{\R^{n+1}_{+}}y^{a+2}x_1^2(\partial_{2}U)^2dxdy,\\
&I_3=\int_{\R^{n+1}_{+}}y^{a+2}x_1x_2\partial_{1}U\partial_{2}Udxdy,\\
&I_4=\int_{\R^{n+1}_{+}}y^{a+4}|\nabla_{x}U|^2dxdy,\\
&I_5=\int_{\R_{+}^{n+1}}y^{a+2}x_1^2(\partial_{y}U)^2dxdy,\\
&I_{6}=\int_{\R_{+}^{n+1}}y^{a+4}|\partial_{y}U|^2dxdy,\\
&I_7=\int_{\R_{+}^{n+1}}y^ax_{1}^2U^2dxdy.
\end{split}\end{equation*}
Then
\begin{equation}\label{I123}
3I_2=3I_3=I_1,
\end{equation}
and
\begin{equation*}\begin{split}
I_3&=\frac{5n^3-10n^2-(a^2-2a+25)n-2a^2+4a+30}{20n(n+2)(n-3)}\int_{\R_{+}^{n+1}}y^{2+a}U^2dxdy,\\
I_{4}&=\frac{(a+3)(5-a)}{5}\int_{\R_{+}^{n+1}}y^{2+a}U^2dxdy,\\
I_5&=\frac{3+a}{20n(3-a)(n-3)}\left(5n^3-30n^2-(a^2+2a-55)n-2a^2+16a-30\right)\\
&\quad\cdot\int_{\R^{n+1}_{+}}y^{a+2}U^2dxdy,\\
I_6&=\frac{(a+5)(a+3)}{5}\int_{\R^{n+1}_{+}}y^{a+2}U^2dxdy,\\
I_{7}&=\frac{3n^2-18n-(a^2-2a-27)}{2(n-3)(3-a)(a+1)}\int_{\R_{+}^{n+1}}y^{a+2}U^2dxdy.
\end{split}\end{equation*}

\end{Lemma}

\begin{proof}
We write here $\zeta=(\xi,\bar{\eta})$ the Fourier variable for $x$, where $\bar{\eta}=(\eta_1,...\eta_{n-1})=(\eta_1,\tilde{\eta})\in \R^{n-1}$ for $\tilde\eta\in\mathbb R^{n-2}$.
We only calculate $I_3$ and $I_4$ here, the rest are very similar.
First, note that
\begin{equation*}\begin{split}
I_{1}=&\int_{0}^{\infty}y^{a+2}\int_{\R^{n-1}}\int_{\R}\left|\partial_{\xi}[\xi\hat{U}(\xi,\bar{\eta},y)]\right|^2d\xi d\bar{\eta}dy,\\
I_{2}=&\int_{0}^{\infty}y^{a+2}\int_{\R^{n-1}}
\int_{\R}|\partial_{\xi}[\eta_{1}\hat{U}(\xi,\bar{\eta},y)]|^2d\xi d\bar{\eta}dy,\\
I_3=&\int_{0}^{\infty}y^{a+2}\int_{\R^n}\xi\eta_1(\partial_{\xi}\hat{U})(\partial_{\eta_1}\hat{U})\,
d\zeta dy,\\
I_4=&\int_{0}^{+\infty}y^{4+a}\int_{\R^n}|\zeta|^2|\hat{U}(\zeta,y)|^2d\zeta dy,\\
I_5=&\int_{0}^{+\infty}y^{2+a}\int_{\R^n}|\partial_{y}\partial_{\xi}\hat{U}(\zeta,y)|^2d\zeta dy,\\
I_6=&\int_{0}^{+\infty}y^{a+4}\int_{\R^n}(\partial_{y}\hat{U}(\zeta,y))^2d\zeta dy,\\
I_7=&\int_{0}^{+\infty}y^{a}\int_{\R^n}|\partial_{\xi}\hat{U}(\zeta,y)|^2d\zeta dy.
\end{split}\end{equation*}
From Lemma \ref{lemma-CS} we can write
\begin{equation}\label{formula40}
\hat{U}(\zeta,y)=\phi(|\zeta|y)\hat{w}(\zeta),
\end{equation}
and one may prove that $\hat U$ is radial in the variable $\zeta$.
Then we compute
\begin{equation*}
\begin{split}
I_1=&\int_0^{+\infty} y^{a+2}\int_{\mathbb R^n} \left[\xi^2|\partial_\xi \hat U|^2
+2\xi\hat U\partial_\xi\hat U+\hat U^2 \right]d\zeta dy\\
=&\int_{0}^{+\infty}y^{a+2}\int_{\R^n}\xi^2|\partial_{\xi}\hat{U}|^2d\zeta dy\\
=&\int_{0}^{+\infty}y^{a+2}\int_{\R^n}\frac{\xi^4}{|\zeta|^2}\hat{U}'^2(|\zeta|,y)d\zeta dy,
\end{split}\end{equation*}
and
\begin{equation*}\begin{split}
I_2=&\int_{0}^{+\infty}y^{a+2}\int_{\R^n}\eta_{1}^{2}|\partial_{\xi}\hat{U}|^2d\zeta dy\\
=&\int_{0}^{+\infty}y^{a+2}\int_{\R^n}\frac{\xi^2\eta_1^2}{|\zeta|^2}\,\hat{U}'^2(|\zeta|,y)d\zeta dy.
\end{split}
\end{equation*}
Next, for any $f(|\zeta|)$ radial function, we define
\begin{equation*}\begin{split}
&X_\alpha:=\int_{\mathbb R^n}\frac{f^2(|\zeta|) \xi^4}{|\zeta|^{\alpha}}\,d\zeta,\\
&Y_\alpha:=\int_{\mathbb R^n}\frac{f^2(|\zeta|) \xi^2\eta_{1}^{2}}{|\zeta|^{\alpha}}\,d\zeta.
\end{split}\end{equation*}
We claim that
\begin{equation}\label{X3Y}X_\alpha=3Y_\alpha.
\end{equation}
Indeed, this is a simple symmetry argument: note that
\begin{equation*}
\int_{\R^n}\frac{f^2(|\zeta|)(\xi+\eta_1)^4}{|\zeta|^{\alpha}}d\zeta=2X_\alpha+6Y_\alpha,
\end{equation*}
and, on the other hand,
\begin{equation*}
\int_{\R^n}\frac{f^2(|\zeta|)(\xi+\eta_1)^4}{|\zeta|^{\alpha}}d\zeta=4X_\alpha,
\end{equation*}
where the last integral is computed using the change of variables
\begin{equation*}
\bar{\xi}=\frac{1}{\sqrt{2}}(\xi+\eta_1), \bar{\eta_1}=\frac{1}{\sqrt{2}}(\xi-\eta_1).
\end{equation*}
Then, from \eqref{X3Y} we immediately obtain that $I_1=3I_2$. The relation with $I_3$ may be computed in a similar way. This shows \eqref{I123}. Finally, the integral $X_\alpha$ may be calculated thanks to
\begin{equation}\label{formulaX}\begin{split}
\int_{\R^n}\frac{\hat{f}^2|\zeta|^4}{|\zeta|^{\alpha}}d\zeta
&=\int_{\R^n}\frac{\hat{f}^2|\xi^2+\eta_{1}^{2}+|\tilde{\eta}|^2|^2}{|\zeta|^{\alpha}}d\zeta\\
&=nX_\alpha+n(n-1)Y_\alpha=\frac{n(n+2)}{3}X_\alpha.
\end{split}\end{equation}

On the other hand, recalling expression \eqref{formula40},
\begin{equation*}\begin{split}
I_1=&\int_{0}^{+\infty}y^{a+2}\int_{\R^n}\xi^2|\partial_{\xi}\hat U|^2d\zeta dy\\
=&\int_{0}^{+\infty}y^{a+2}\int_{\R^n}\xi^2|\partial_{\xi}\hat{w}|^2\cdot\phi^2(|\zeta|y)\, d\zeta dy\\
&+2\int_{0}^{+\infty}y^{a+3}\int_{\R^n}
\frac{\xi^3}{|\zeta|}\hat{w}\partial_{\xi}\hat{w}
\phi(|\zeta|y)\phi^{'}(|\zeta|y)
d\zeta dy\\
&+\int_{0}^{+\infty}y^{a+4}\int_{\R^n}
\frac{\xi^4}{|\zeta|^2}\hat{w}^2\phi'^2(|\zeta|y)\,d\zeta dy\\
=:&\,H_1+H_2+H_3.
\end{split}\end{equation*}
Direct calculation shows that, after the change of variables $s=|\zeta|y$,
\begin{equation*}\begin{split}
H_1=&\int_{0}^{+\infty}s^{a+2}\phi^{2}(s)\,ds\int_{\R^n}\frac{\xi^2}
{|\zeta|^{a+3}}\,|\partial_{\xi}\hat{w}|^2d\zeta,\\
H_2=&\int_{0}^{+\infty}s^{a+3}\phi(s)\cdot\phi^{'}(s)ds\int_{\R^n}\frac{
(a+2)\xi^4-3\xi^2|\bar{\eta}|^2}
{|\zeta|^{a+7}}\,\hat{w}^2\,d\zeta,\\
H_3=&\int_{0}^{+\infty}s^{a+4}{\phi'}^{2}(s)\,ds\int_{\R^n}\frac{\xi^4}
{|\zeta|^{a+7}}\,\hat{w}^2d\zeta.
\end{split}\end{equation*}
Thus, using \eqref{X3Y} and \eqref{formulaX} for $X_\alpha$,
\begin{equation*}\begin{split}
H_2&=(a-n+3)\int_{0}^{+\infty}s^{a+3}\phi(s)\phi'(s)\,ds
\int_{\R^n}\frac{\hat{w}^2\xi^4}{|\zeta|^{a+7}}\,d\zeta\\
&=\frac{3(a-n+3)}{n(n+2)}\int_{0}^{+\infty}s^{a+3}\phi(s)\phi'(s)\,ds
\int_{\R^n}\frac{\hat{w}^2}{|\zeta|^{a+3}}\,d\zeta\\
&=\frac{3(n-a-3)(a+3)}{2n(n+2)}\int_{0}^{+\infty}s^{a+2}\phi^2(s)\,ds
\int_{\R^n}\frac{\hat{w}^2}{|\zeta|^{a+3}}\,d\zeta\\
&=\frac{3(n-a-3)(a+3)}{2n(n+2)}\int_{0}^{+\infty}y^{a+2}
\int_{\R^n}\hat{U}^2\,d\zeta dy,\\
\end{split}\end{equation*}
and again, thanks to \eqref{formulaX}, and the relation between $\phi$ and $\phi'$ given in \eqref{7},
\begin{equation*}\begin{split}
H_3&=\frac{1}{3}\int_{0}^{+\infty}s^{a+4}{\phi'}^2ds
\int_{\R^n}\frac{\hat{w}^2\xi^4}{|\zeta|^{a+7}}\,d\zeta\\
&=\frac{1}{3n(n+2)}\int_0^{+\infty}s^{a+4}\phi'(s)^2\int_{\mathbb R^n} \frac{\hat w^2}{|\zeta|^{a+3}}\,d\zeta\\
&=\frac{(a+5)(a+3)}{5n(n+2)}\int_{0}^{+\infty}y^{a+2}\int_{\R^n}U^2dxdy.
\end{split}\end{equation*}

Next, we give the estimate for $H_1$. Note that for $w$ as given \eqref{diffeo-1} we have that
\ba\label{omega}
\hat{w}(\zeta)=C_0|\zeta|^{-\gamma}K_{\gamma}(|\zeta|),
\ea
where $K_{\gamma}(s)$ is the modified Bessel function from Lemma \ref{bessel0}. This is a well known formula for which we have not found a proof, so we provide one in the Appendix. We have that
\begin{equation*}
\partial_{\xi}\hat{w}=C_{0}\left[-\gamma|\zeta|^{-\gamma-1}K_{\gamma}(|\zeta|)
+|\zeta|^{-\gamma}K'_{\gamma}(|\zeta|)\right]\cdot\frac{\xi}{|\zeta|}.
\end{equation*}
We may calculate directly from  \eqref{formulaX},
\begin{equation}\label{equationH1}\begin{split}
H_1&=\frac{3C_0^2}{n(n+2)}\int_{0}^{+\infty}s^{a+2}\phi^{2}(s)
\int_{\R^n}\frac{\left[K'_{\gamma}(|\zeta|)-\frac{\gamma K_{\gamma}(|\zeta|)}{|\zeta|}\right]^2}{|\zeta|^2}\,d\zeta\\
&=\frac{3C_0^2|\mathbb S_{n-1}|}{n(n+2)}\int_{0}^{+\infty}s^{a+2}\phi^{2}(s)
\int_{0}^{+\infty}t^{n-3}(K'_{\gamma}(t)-\gamma K_{\gamma}(t)/t)^2\,dt.
\end{split}\end{equation}
Note that $K_\gamma(t)$ is a solution of \eqref{bessel}. Thus, multiplying this equation by $t^{n-4}K_{\gamma}'(t)$ and integrating we arrive at
\begin{equation*}\begin{split}
0&=\int_{0}^{+\infty}t^{n-2}K_{\gamma}'K_{\gamma}''dt+\int t^{n-3}{K_{\gamma}'}^2\,dt
-\int_{0}^{\infty}t^{n-2}K_{\gamma} K_{\gamma}'dt-\gamma^2\int_{0}^{+\infty}t^{n-4}K_{\gamma}K_{\gamma}'\,dt\\
&=-\tfrac{n-2}{2}\int_{0}^{+\infty}t^{n-3}{K_{\gamma}'}^2dt+\int_{0}^{+\infty}t^{n-3}{K_{\gamma}'}^2\,dt
+\tfrac{n-2}{2}\int_{0}^{+\infty}t^{n-3}K_{\gamma}^2\,dt
\\&+\tfrac{\gamma^2(n-4)}{2}\int_{0}^{+\infty}t^{n-5}K_{\gamma}^2\,dt,
\end{split}\end{equation*}
from which we get
\begin{equation*}
\int_{0}^{+\infty}t^{n-3}{K_{\gamma}'}^2\,dt=\frac{1}{n-4}
\left\{(n-2)\int_{0}^{+\infty}t^{n-3}K_{\gamma}^2\,dt
+\gamma^2(n-4)\int_{0}^{+\infty}t^{n-5}K_{\gamma}^2\,dt \right\}.
\end{equation*}
Expanding out \eqref{equationH1}, taking into account the above expression we arrive at
\begin{equation*}\begin{split}
H_1
=&\frac{3C_0^2|S_{n-1}|}{n(n+2)}\int_{0}^{+\infty}s^{a+2}\phi^2(s)\,ds\\
&\cdot\left(\frac{n-2}{n-4}\int_{0}^{+\infty}t^{n-3}K_\gamma^2\,dt+(2\gamma^2+\gamma(n-4))
\int_{0}^{+\infty}t^{n-5}K_\gamma^2\,dt\right)\\
=&\frac{3|\mathbb S^{n-1}|C_0^2}{n(n+2)}
\left(\frac{(n-2)(n-5+a)(n-3-a)}{4(n-3)}+\gamma(n-4)+2\gamma^2\right)\\
\quad &\cdot\int_{0}^{+\infty}s^{a+2}\phi^2(s)\,ds\int_{0}^{+\infty}t^{n-5}K_\gamma^2\,dt,
\end{split}\end{equation*}
where we have used \eqref{9} to combine both terms above, since $\phi(t):=t^{\gamma}K_{\gamma}(t)$ satisfies \eqref{ode}. Finally, since
\begin{equation*}
\begin{split}
|\mathbb S^{n-1}|C_0^2
\int_{0}^{+\infty}s^{a+2}\phi^2(s)\,ds\int_{0}^{+\infty}t^{n-5}K_\gamma^2(t)dt
&=\int_0^{+\infty}s^{a+2}\phi^2(s)\,ds \int_{\mathbb R^n}  \frac{\hat w^2(|\zeta|)}{|\zeta|^{a+3}}\,d\zeta\\
&=\int_{\mathbb R^{n+1}_+} y^{a+2} U^2dxdy,
\end{split}\end{equation*}
we obtain a formula for $H_1$
$$H_1=\frac{2}{n(n+2)}
\left(\frac{(n-2)(n-5+a)(n-3-a)}{4(n-3)}+\gamma(n-4)+2\gamma^2\right)
\int_{\R^{n+1}_{+}}y^{a+2}U^2dxdy.$$
From here we can calculate $I_1,I_2,I_3$ and, in particular,
\begin{equation*}\begin{split}
I_3=&\left(\frac{5n^2+5(a+1)n+4a^2-18a-10}{20n(n+2)}-\frac{a(n-2)(n-5+a)}{4n(n-3)(n+2)}\right)
\int_{\R_{+}^{n+1}}y^{2+a}U^2dxdy\\
=&\frac{5n^3-10n^2-(a^2-2a+25)n-2a^2+4a+30}{20n(n+2)(n-3)}\int_{\R_{+}^{n+1}}y^{2+a}U^2 dxdy.
\end{split}\end{equation*}
Similarly,
\begin{equation*}
I_4=\int_{0}^{+\infty}{s^{4+a}\phi^{2}(s)}ds\int_{\R^n}\frac{|\hat{w}|^2(\zeta)}{|\zeta|^{3+a}}d\zeta
=\frac{(a+3)(5-a)}{5}\int_{\R_{+}^{n+1}}y^{a+2}U^{2}dxdy.
\end{equation*}
\end{proof}

\section{Proof of Theorem \ref{theorem2}}

By the work of \cite{gonzalezqing}, it is enough to find a suitable test function such that inequality  \eqref{condition-Lambda} is strictly satisfied.

On $\R^n$, we fix the conformal diffeomorphisms of the sphere
\begin{equation*}
w_{\mu}(x):=\left(\frac{\mu}{|x|^2+\mu^2}\right)^{\frac{n-2\gamma}{2}},
\end{equation*}
which satisfy
\begin{equation}\label{diffeo}
(-\Delta)^\gamma w_\mu=c\, w_\mu^{\frac{n+2\gamma}{n-2\gamma}},
\end{equation}
for some positive constant $c$. We also consider the corresponding extension $U_{\mu}:=U(w_{\mu})$ from Lemma \ref{lemma-CS}, that can be written as
\begin{equation*}
U_{\mu}(x,y)=\mathcal K_{\gamma}\ast_{x}w_{\mu}.
\end{equation*}
 It is clear that
\begin{equation*}
w_{\mu}(x)=\frac{1}{\mu^{\frac{n-2\gamma}{2}}}w_{1}\left(\frac{x}{\mu}\right),
~~\mbox{and}~U_{\mu}(x,y)=\frac{1}{\mu^{\frac{n-2\gamma}{2}}}
U_1\left(\frac{x}{\mu},\frac{x}{\mu}\right).
\end{equation*}
These functions attain the best constant in the trace Sobolev inequality \eqref{Euclidean-Sobolev}. More precisely, looking at \eqref{Sobolev-sphere},
\begin{equation}\label{embedding}
\|w_\mu\|_{L^{2^*}(\mathbb R^n)}^2= \bar S(n,\gamma)\int_{\mathbb R^{n+1}_+}
y^a |\nabla {U_\mu}|^2 \,dxdy.\end{equation}
From \eqref{diffeo} we know that $U_\mu$ is the
(unique) solution of the problem
\begin{equation}\label{extension-Rn}\left\{\begin{split}
\divergence(y^a \nabla U_\mu)=0 &\quad\mbox{in } \mathbb R^{n+1}_+,\\
-\lim_{y\to 0}y^a\partial_y
U_\mu=c_{n,\gamma}(w_\mu)^{\frac{n+2\gamma}{n-2\gamma}}&\quad\mbox{on
} \mathbb R^n,
\end{split}\right.\end{equation}
On the other hand, if we multiply equation \eqref{extension-Rn} by
$U_\mu$ and integrate by parts,
\begin{equation}\label{formula200}\int_{\mathbb R^{n+1}_+} y^a|\nabla U_\mu|^2\,dxdy=
c_{n,\gamma} \int_{\mathbb R^n} (w_\mu)^{2^*}\,dx.\end{equation}
Now we compare
\eqref{formula200} with \eqref{embedding}. Using
\eqref{Sobolev-sphere} we arrive at
\begin{equation}\label{formula202}\Lambda_\gamma(\mathbb S^n,[g_c])=c_{n,\gamma}d_\gamma^*\left[\int_{\mathbb R^n}
(w_\mu)^{2^*}dx\right]^{\frac{2\gamma}{n}}.\end{equation}

Note that $w_{\mu}$ is also radially symmetric and nonincreasing, so also $U_{\mu}=\mathcal K_{\gamma}\ast_{x}w_{\mu}$ is radially symmetric and non-increasing since the kernel $\mathcal K_{\gamma}$ is as such.

Given any $\epsilon>0$, let $B_\epsilon$ be the ball of radius
$\epsilon$ centered at the origin in $\mathbb R^{n+1}$ and
$B_\epsilon^+$ be the half ball of radius $\epsilon$ in $\mathbb
R^{n+1}_+$. Choose a smooth radial cutoff function $\eta$, $0\leq
\eta\leq 1$, supported on $B_{2\epsilon}$, and satisfying $\eta=1$
on $B_\epsilon$. For $\mu<<\epsilon$, we choose as test function simply
$$V_\mu:=\eta U_\mu,$$
for the functional \eqref{barI}, which we recall is given by
\begin{equation}\label{energy2}
\overline I_{\gamma}[V,\hat h]=\frac{d_{\gamma}^{*}\int_{X^{n+1}}\left({\rho}^a|\nabla V|_{\bar{g}}^2+E(\rho)V^2\right)\,dv_{\bar{g}}}
{(\int_{M^n}V^{2^{*}}\,dv_{\hat{h}})^{2/{2^{*}}}}
\end{equation}

\noindent{\bf Step 1: Computation of the Energy in $B_{\epsilon}^{+}$.}

\noindent Here $V_\mu=U_\mu$. By Proposition \ref{metric1}, using the expansion for  $\sqrt{\det \bar g}$,
\begin{equation}\label{equation100}\begin{split}
\int_{B_{\epsilon}^{+}}y^{a}|\nabla U_{\mu}|_{\bar{g}}^2dv_{\bar{g}}&=\int_{B_{\epsilon}^{+}}y^a[{\bar g}^{ij}\partial_{i}U_{\mu}(\partial_{j}U_{\mu})+(\partial_{y}U_{\mu})^2] dv_{\bar{g}}\\
&=\int_{B_{\epsilon}^{+}}y^a |\nabla U_{\mu}|_{\bar{g}}^2dxdy+\int_{B_{\epsilon}^{+}}y^a|\nabla U_{\mu}|_{\bar{g}}^2\cdot O(|(x,y)|^5)dxdy\\
&-\frac{1}{6}\overline{Ric}_{yy,y}\int_{B_{\epsilon}^{+}}y^{a+3}|\nabla U_{\mu}|_{\bar{g}}^2dxdy
+\frac{1}{2}\int_{B_{\epsilon}^{+}}y^{a}|\nabla U_{\mu}|_{\bar{g}}^2(\det \bar{g})^{(4)}dxdy\\
&=I+II+III+IV,
\end{split}\end{equation}
where $(\det \bar{g})^{(4)}$ means the fourth order $O(r^4)$, $r=|(x,y)|$, in the expansion of $\det \bar{g}$.

As to $II$,
\begin{equation}\label{II}\begin{split}
\int_{B_{\epsilon}^{+}}y^a|\nabla U_{\mu}|_{\bar{g}}^2\cdot O(|(x,y)|^5)dxdy&\le C\mu^5\int_{B_{\epsilon/\mu}^{+}}y^{a}|(x,y)|^5|\nabla U_1|^2dxdy\\
&=\mu^5[\tilde{\cal{E}}_5+o(1)],
\end{split}\end{equation}
where
$f=o(1)$ means that
$$
\lim_{\epsilon/\mu\rightarrow \infty}f=0,
$$
and
$$\tilde{\cal{E}}_{k}:=\int_{\R_{+}^{n+1}}y^{a}|(x,y)|^k|\nabla U_1|^2dxdy.$$
From \eqref{omega} and expansion formula for $K_{\gamma}$ in Lemma \ref{bessel0},  it is easy to check that both $${\cal{E}}_{k}:=\int_{\R_{+}^{n+1}}y^{a+k}|\nabla U_1|^2dxdy<+\infty$$
and
$$\int_{\R_{+}^{n+1}}y^a|x|^5|\nabla U_1|^2dxdy<+\infty$$
are finite when $n>5+2\gamma $, as a consequence, $\tilde {\mathcal E}_5<\infty$.\\

Next, we estimate the term $III$,
\begin{equation*}\begin{split}
III&=-\frac{1}{6}\overline{Ric}_{yy,y}\mu^3\int_{B_{\epsilon/\mu}^{+}}y^{a+3}|\nabla U_1|^2dxdy-\frac{1}{6}\overline{Ric}_{yy,y}\int_{B_{\epsilon}^{+}}y^{a+3}(\bar g^{ij}
-\delta_{i}^{j})(\partial_iU_{1})(\partial_{j}U_{1})dxdy\\
&=-\frac{1}{6}\bar{R}_{yy,y}\mu^3\int_{B_{\epsilon/\mu}^{+}}y^{a+3}|\nabla U_1|^2dxdy\\
&+\mu^5\left(\int_{\R_{+}^{n+1}}y^{a+3}O(|(x,y)|^2)|\nabla U_1|^2dxdy+o(1)\right).
\end{split}\end{equation*}
While using \eqref{10} from Lemma \ref{bessel1}, recalling \eqref{hatU} and the change $s=|\zeta|y$,
\begin{equation*}\begin{split}
\int_{\R_{+}^{n+1}}y^{a+3}|\nabla U_1|^2dxdy&=\int_{\R^n}\frac{|\hat{w}|^2(\zeta)}{|\zeta|^{a+2}}\int_{0}^{+\infty}s^{a+3}\left(\phi^2(s)
+{\phi'}^2(s)\right)ds\\
&=\frac{3(a+2)}{2}\int_{\R^n}\frac{|\hat{w}|^2(\zeta)}{|\zeta|^{a+2}}\int_{0}^{+\infty}s^{a+1}\phi^2(s)ds\\
&=\frac{3(a+2)}{2}\int_{\R_{+}^{n+1}}y^{a+1}U_1^2(x,y)dxdy.
\end{split}\end{equation*}
Thus, for $n>5+2\gamma$,
\ba\label{III}
III&=&-\frac{(a+2)}{4}\overline{Ric}_{yy,y}\mu^3\int_{\R_{+}^{n+1}}y^{a+1}U_1^2(x,y)dxdy+O(\mu^5).
\ea

Now, we give an estimate $IV$. But, noting the symmetry property for the curvature and integral,
\begin{equation*}\begin{split}
2\,IV&=\int_{B_{\epsilon}^{+}}y^a|\nabla U_{\mu}|_{\bar{g}}^2(\det \bar{g})^{(4)}dxdy\\
&=\int_{B_{\epsilon}^{+}}y^a(\det \bar g)^{(4)}|\nabla U_{\mu}|^2dxdy+\int_{B_{\epsilon}^{+}}y^a(\det g)^{(4)}(\bar{g}^{ij}-\delta_{ij})(\partial_{i}U_{\mu})(\partial_{j}U_{\mu})\,dxdy\\
&=\int_{B_{\epsilon}^{+}}y^a(\det \bar{g})^{(4)}|\nabla U_{\mu}|^2dxdy+O(\mu^5)\\
&=\left(-\frac{1}{20}{Ric}_{kl,mp}-\frac{1}{90}\bar{R}_{iksl}R_{imsp}\right)
\int_{B_{\epsilon}^{+}}y^ax_kx_lx_mx_p|\nabla U_{\mu}|^2dxdy\\\
&\quad-\frac{1}{2}\sum_{k}\overline{Ric}_{yy,kk}\int_{B_{\epsilon}^{+}}y^ax_{1}^{2}|\nabla U_{\mu}|^2dxdy\\
&\quad-\frac{1}{12}\left(\overline{Ric}_{yy,yy}+2(\bar{R}_{yiyj})^2\right)\int_{B_{\epsilon}^{+}}y^{a+4}|\nabla U_{\mu}|^2dxdy\\
&\quad+O(\mu^5)\\
&=IV.1+IV.2+IV.3+O(\mu^5)
\end{split}\end{equation*}
Proposition \ref{metric1}\emph{(vi)} immediately gives that $IV.1=0$. Next, we estimate IV.2. For that, we write
\begin{equation*}
IV.2=-\frac{1}{2}\mu^4\sum_{k=1}^{n}\overline{Ric}_{yy,kk}(IV.2.1+(n-1)IV.2.2+IV.2.3)+\mu^4o(1),
\end{equation*}
where
\begin{equation*}\begin{split}
&IV.2.1=\int_{0}^{+\infty}\int_{\R^n}y^{a+2}x_1^2(\partial_1U_1)^2dxdy,\\
&IV.2.2=\int_{0}^{+\infty}\int_{\R^n}y^{a+2}x_2^2(\partial_1U_1)^2dxdy,\\
&IV.2.3=\int_{0}^{+\infty}\int_{\R^n}y^{a+2}x_1^2(\partial_{y}U_1)^2dxdy.
\end{split}\end{equation*}
Then Lemma \ref{technical} quickly yields that
\begin{equation*}
IV.2=-\frac{1}{2}\overline{Ric}_{yy,kk}\mu^4\left((n+2)I_3+I_5\right).
\end{equation*}
Finally, using the notation from the same lemma, we can write the term $IV.3$ as
\begin{equation*}
IV.3=-\frac{1}{12}(\overline{Ric}_{yy,yy}+2(\bar{R}_{yiyj})^2)(I_4+I_6)\mu^4+\mu^4o(1).
\end{equation*}
Thus, putting all together we arrive at
\begin{equation}\label{IV}
IV=\left\{-\frac{1}{4}\overline{Ric}_{yy,kk}\left((n+2)I_3+I_5\right)
-\frac{1}{24}\left(\overline{Ric}_{yy,yy}+2(\bar{R}_{yiyj})^2\right)(I_4+I_6)\right\}\mu^4+\mu^4o(1),
\end{equation}
for $n>5+2\gamma$.

To conclude, we give the estimate for the term $I$ in \eqref{equation100}. Direct calculation shows that
$$x_1\partial_{2}U_{1}=x_2\partial_{1}U_{1}.$$
In fact, the Fourier transform of $x_1\partial_{2}U_{1}$ is
\begin{equation*}\begin{split}
\partial_{\xi}(\eta_1\hat{U}_1(\xi,\eta_{1},\tilde{\eta}))
&=\eta_1\partial_{\xi}(\phi(|\zeta|y)\hat{w}_1(\zeta))\\
&=\eta_1\left[(\partial_\xi \hat w_1)\phi(|\zeta|y)+y\frac{\xi}{|\zeta|}\hat w_1\phi'(|\zeta|y)\right],
\end{split}\end{equation*}
and we have that $x_1 \partial_2 w_1=x_2\partial_1 w_1$. Thus the previous expression is symmetric with respect to the first two variables.

Then an analogous symmetry argument yields that we may restrict to consider the fourth order terms in the expansion of $\bar{g}^{ij}$. Thus
\begin{equation}\label{formula110}\begin{split}
I&=\int_{B_{\epsilon}^{+}}y^a|\nabla U_{\mu}|^2dxdy+\int_{B_{\epsilon}^{+}}y^{a}(g^{ij}-\delta^{ij})(\partial_i U_{\mu})(\partial_{j}U_{\mu})dxdy\\
&=\int_{B_{\epsilon}^{+}}y^{a}|\nabla U_{\mu}|^2dxdy\\
&+\int_{B_{\epsilon}^{+}}\left(\frac{1}{24}{\bar g^{ij}}_{,yyyy}y^4+\frac{1}{4}\bar g^{ij}_{,yykl}y^2x_kx_l
+\frac{1}{24}\bar g_{,klmp}^{ij}x_kx_lx_mx_p\right)(\partial_iU_{\mu})(\partial_jU_{\mu})   dxdy\\
&=\int_{B_{\epsilon}^{+}}y^{a}|\nabla U_{\mu}|^2dxdy+A_1+A_2+A_3.
\end{split}\end{equation}
First we estimate $A_1$, given by
$$24A_1=\int_{B_{\epsilon}^{+}}\bar g_{,yyyy}^{ij}y^{4+a}(\partial_{i}U_{\mu})(\partial_{j}U_{\mu})dxdy.$$
Note that
$$\partial_{i}U_{\mu}=-C_{n,\gamma}(n+1-a)\int_{\R^n}\frac{y^{1-a}(x_i-\tilde x_i)}{(|x-\tilde x|^2+y^2)^{\frac{n+1-a}{2}+1}}
\cdot\left(\frac{\mu}{|\tilde x|^2+\mu^2}\right)^{\frac{n-2\gamma}{2}}d\tilde x.$$
Then, because of expression \eqref{expansioninverseg} for the inverse of the metric we have that
\begin{equation*}\begin{split}
24A_1&=\int_{B_{\epsilon}^{+}}(2\bar{R}_{yiyj,yy}
+16\bar{R}_{yiys}\bar{R}_{ysyj})y^{4+a}(\partial_iU_{\mu})(\partial_jU_{\mu})dxdy\\
&=\frac{1}{n}\int_{B_{\epsilon}^{+}}(2\overline{Ric}_{yy,yy}+16(\bar{R}_{yiyj})^2)y^{4+a}|\nabla_{x} U_{\mu}|^2dxdy\\
&=\frac{2\overline{Ric}_{yy,yy}+16(\bar{R}_{yiyj})^2}{n}\mu^{4}\int_{B_{\epsilon/\mu}^{+}}y^{4+a}|\nabla_x U_1|^2(x,y)dxdy.
\end{split}\end{equation*}
Also, using again expression \eqref{expansioninverseg}, we have that $$\bar{g}^{ij}_{,yykl}=2\bar{R}_{yiyj,kl};$$
(the other terms vanish thanks to \eqref{2.17}). Thus, by a symmetry argument,
\begin{equation*}\begin{split}
4A_2&=\int_{B_{\epsilon}^{+}}\bar g_{,yykl}^{ij}y^{2+a}x_kx_l(\partial_i U_{\mu})(\partial_j U_{\mu})dxdy\\
&=2\bar{R}_{yiyi,ii}\int_{B_{\epsilon}^{+}}y^{2+a}x_{i}^2(\partial_{i}U_{\mu})^2dxdy\\
&+\sum_{i\neq j}2\bar{R}_{yiyi,jj}\int_{B_{\epsilon}^{+}}y^{2+a}x_{i}^{2}(\partial_{j}U_{\mu})^2dxdy\\
&+\sum_{i\neq j}4\bar{R}_{yiyj,ij}\int_{B_{\epsilon}^{+}}y^{2+a}x_ix_j(\partial_{i}U_{\mu})(\partial_{j}U_{\mu})dxdy.
\end{split}\end{equation*}
After changing variables and reordering, taking into account \eqref{I123} to group some of the terms,
\begin{equation*}\begin{split}
4A_2&=2\mu^{4}\bar{R}_{yiyi,ii}\int_{B_{\epsilon/\mu}^{+}}y^{2+a}x_{i}^2(\partial_{i}U_1)^{2}dxdy\\
&+2\mu^{4}\sum_{i\neq j}\bar{R}_{yiyi,jj}\int_{B_{\epsilon/\mu}^{+}}y^{a+2}x_{i}^{2}(\partial_{j}U_1)^2dxdy\\
&+4\mu^{4}\sum_{i\neq j}\bar{R}_{yiyj,ij}\int_{B_{\epsilon/\mu}^{+}}y^{2+a}x_ix_j(\partial_{i}U_1)(\partial_{j}U_1)dxdy\\
&=2\mu^{4}\Big[\bar{R}_{yiyi,ii}I_1
+\sum_{i\neq j}\bar{R}_{yiyi,jj}I_2
+2\sum_{i\neq j}\bar{R}_{yiyj,ij}I_3+o(1)\Big]\\
&=2\mu^{4}\Big[\Big\{\bar{R}_{yiyi,ii}+2\bar{R}_{yiyi,ii}+\sum_{i\neq j}\bar{R}_{yiyi,jj}
+2\sum_{i\neq j}\bar{R}_{yiyj,ij}\Big\}I_3+o(1)\Big].
\end{split}\end{equation*}
 We conclude that
$$A_2=\frac{\mu^4}{2}\left(\sum_{k}\overline{Ric}_{yy,kk}+2\sum_{i,j}\bar{R}_{yiyj,ij}\right)\left[
\int_{B_{\epsilon/\mu}^{+}}y^{2+a}x_1x_2(\partial_{1}U_1)(\partial_{2}U_1)dxdy+o(1)\right].$$
As to the term $\int_{B_{\epsilon}^{+}}y^a|\nabla U_{\mu}|^2dxdy$ in expression \eqref{formula110}, we use the equation \eqref{extension-Rn} to get
\begin{equation*}\begin{split}
\int_{B_{\epsilon}^{+}}y^{a}|\nabla U_{\mu}|^2dxdy&=\int_{\Gamma_{\epsilon}^{+}}y^aU_{\mu}\partial_{\nu}U_{\mu}d\sigma
-\int_{\Gamma_{\epsilon}^{0}} \lim_{y\to 0} U_\mu y^a\partial_y U_\mu\\
&\le  c_{n,\gamma}\int_{\Gamma_{\epsilon}^{0}} w_{\mu}^{2^*}dx\le\frac{\Lambda_\gamma(\mathbb S^{n},[g_c])}{d_{\gamma}^{*}}
\left(\int_{\Gamma_{\epsilon}^{0}}w_{\mu}^{2^{*}}dx\right)^{\frac{n-2\gamma}{n}},
\end{split}\end{equation*}
where we have used that $\partial_{\nu}U_{\mu}\le 0$ on $\Gamma_{\epsilon}^{+}$ and \eqref{formula202}. In addition, the third term $A_3$ vanishes due to the symmetries of the curvature tensor.
Thus \eqref{formula110} reduces to
\begin{equation}\label{I}\begin{split}
I&\le  \frac{\Lambda_\gamma(\mathbb S^n,[g_c])}{d_{\gamma}^{*}}
\left(\int_{\Gamma_{\epsilon}^{0}}w_{\mu}^{2^{*}}dx\right)^{\frac{n-2\gamma}{n}}\\
&+\left\{\frac{1}{12n}(\overline{Ric}_{yy,yy}+8(\bar{R}_{yiyj})^2)I_{4}
+\frac{1}{2}(\overline{Ric}_{yy,kk}+2\sum\bar{R}_{yiyj,ij})I_3\right\}\mu^4+\mu^4o(1).
\end{split}\end{equation}

Finally, we can give an estimate for the energy \eqref{equation100}. Putting together  \eqref{I}, \eqref{II}, \eqref{III} and \eqref{IV} we conclude that
\begin{equation}\label{1}\begin{split}
\int_{B_{\epsilon}^{+}}y^a|\nabla V_{\mu}|^2dv_{\bar{g}}
&\le\frac{\Lambda_\gamma(\mathbb S^n,[g_c])}{d_{\gamma}^{*}}
\left(\int_{\Gamma_{\epsilon}^{0}}w_{\mu}^{2^*}\right)^{\frac{n-2\gamma}{n}}
-\frac{(a+2)}{4}\overline{Ric}_{yy,y}\mu^3\int_{B_{\epsilon/\mu}^{+}}y^{a+1}U_1^2dxdy\\
&+\left\{-\frac{1}{4}\overline{Ric}_{yy,kk}((n+2)I_3+I_5)
-\frac{1}{24}\overline{Ric}_{yy,yy}(I_4+I_6)\right.\\
&~~~~~+\left.\frac{1}{12n}\overline{Ric}_{yy,yy}I_4
+\frac{1}{2}(\overline{Ric}_{yy,kk}+2\bar{R}_{yiyj,ij})I_3\right\}\mu^4\\
&+o(\mu^4),
\end{split}\end{equation}
for $n>5+2\gamma$. Here we have used property \eqref{2.17} of the metric to cancel the terms $\bar R_{yiyj}$ in the integrals $I$ and $IV$.\\

On the other hand, now we calculate the term $\int_{B_{\epsilon}^{+}}E(y)U_{\mu}^2dv_{\bar{g}}$ in the energy \eqref{energy2}. For a metric $g^+=\rho^{-2}(d\rho^2+h_\rho)$ we may explicitly calculate its Laplace-Beltrami operator, and thus,
 \begin{equation*}\begin{split}
 E(\rho)=&\rho^{-1-s}(-\Delta_{g^{+}}-s(n-s))\rho^{n-s}\\
  &=-\frac{n-s}{2}\rho^{n-2s}\frac{\partial_{\rho}\det h_{\rho}}{\det h_{\rho}}=-\frac{n-1+a}{4}\rho^{a-1}\frac{\partial_{\rho}\det h_{\rho}}{\det h_{\rho}}.
 \end{split}\end{equation*}
 We need to calculate the expansion for $\frac{\partial_{y}\det h_{y}}{\sqrt{\det h_{y}}}$ near $p$. But
  $\det h_{y}=\det\bar{g}$, thus substituting the expansion \eqref{detg} we arrive at
  $$\frac{\partial_{y}\det h_{y}}{\sqrt{\det h_{y}}}=1-\overline{Ric}_{yy,kl}x_k x_l\rho-\overline{Ric}_{yy,y}\rho^2+\frac{1}{3}[-\overline{Ric}_{yy,yy}-2(\bar R_{y iy j})^2]+\ldots,$$
 where we have not written terms that will integrate to zero, in particular because of statement \emph{(vi)} in Proposition \ref{metric1}. Then, noting that $dv_{\bar g}=\sqrt{\det h_y}\,dxdy$, we have
  \begin{equation}\label{2}\begin{split}
 \int_{B_{\epsilon}^{+}}E(y)U_{\mu}^{2}dv_{\bar{g}}&= -\frac{n-1+a}{4}\int_{B_{\epsilon}^{+}}y^{a-1}\left(\frac{\partial_{y}\det h_{y}}{\sqrt{\det h_{y}}}\right)U_\mu^2\, dxdy\\
 &\le \frac{n-1+a}{4}\overline{Ric}_{yy,y}\int_{B_{\epsilon}^{+}}y^{a+1}U_{\mu}^{2}dxdy\\
 &\quad+\frac{n-1+a}{4}\overline{Ric}_{yy,kl}\int_{B_{\epsilon}^{+}}y^{a}x_k x_l U_{\mu}^2dxdy\\
&\quad+\frac{n-1+a}{12}\left(\overline{Ric}_{yy,yy}+2(\bar{R}_{yiyj})^2\right)
\int_{B_{\epsilon}^{+}}y^{a+2}U_{\mu}^2dxdy\\
&\quad+C\int_{B_{\epsilon}^{+}}y^a|(x,y)|^3U_{\mu}^2dxdy\\
&\le\frac{n-1+a}{4}\overline{Ric}_{yy,y}\mu^3\int_{B_{\epsilon/\mu}^{+}}y^{a+1}U_{1}^{2}dxdy\\
&\quad+\frac{n-1+a}{4}\overline{Ric}_{yy,kk}\mu^4\int_{B_{\epsilon/\mu}^{+}}y^ax_1^2U_{1}^{2}dxdy\\
&\quad+\frac{n-1+a}{12}\mu^4\left(\overline{Ric}_{yy,yy}+2(\bar{R}_{yiyj})^2\right)
\int_{B_{\epsilon/\mu}^{+}}y^{a+2}U_{1}^{2}dxdy\\
&\quad+C\mu^5(\tilde{\mathcal E}_3+o(1)).
 \end{split}\end{equation}

\noindent {\bf{Step 2: Computation of the energy in the half-annulus $B_{2\epsilon}^{+}\setminus B_{\epsilon}^{+}$}.}

\noindent At first, we note that on the half-annulus,
\begin{equation}\label{formula300}|\nabla V_{\mu}|_{\bar{g}}^{2}\le c|\nabla U_{\mu}|^2+\frac{c}{\epsilon^2}(U_{\mu})^2.\end{equation}
But
\begin{equation*}\begin{split}
\int_{B_{2\epsilon}^{+}\setminus B_{\epsilon}^{+}}y^a(U_{\mu})^2dxdy
&\le {\mu^2}\int_{B_{2\epsilon/\mu}^{+}\setminus B_{\epsilon/\mu}^{+}}y^a(U_{1})^2dxdy\\
&\le {\mu^2}\left(\frac{\epsilon}{\mu}\right)^{-3}\int_{B_{2\epsilon/\mu}^{+}\setminus B_{\epsilon/\mu}^{+}}y^a|(x,y)|^3U_{1}^2dxdy\\
&\le\mu^5\epsilon^{-3}o(1),
\end{split}\end{equation*}
and
\begin{equation*}\begin{split}
\int_{B_{2\epsilon}^{+}\setminus B_{\epsilon}^{+}}y^a|\nabla U_{\mu}|^2dxdy
&=\int_{B_{2\epsilon/\mu}^{+}\setminus B_{\epsilon/\mu}^{+}}y^a|\nabla U_1|^2dxdy\\
&\le\mu^5\epsilon^{-5}o(1).
\end{split}\end{equation*}
Thus from formula \eqref{formula300} we may estimate
\begin{equation}\label{3}
\int_{B_{2\epsilon}^{+}\setminus B_{\epsilon}^{+}}y^a|\nabla V_{\mu}|_{\bar{g}}^2\,dv_{\bar{g}}\le c\mu^5\epsilon^{-5}o(1).
\end{equation}
And similarly,
\begin{equation}\label{4}
\int_{B_{2\epsilon}^{+}\setminus B_{\epsilon}^{+}}E(y)V_{\mu}^2\,dv_{\bar{g}}\le c\mu^5\epsilon^{-3}o(1).
\end{equation}

\noindent{\bf{Step 3: Conclusion.}}

\noindent Next, \eqref{1}, \eqref{2}, \eqref{3} and \eqref{4} show that
\begin{equation*}\begin{split}
\int_{X}y^a&|\nabla V_{\mu}|_{\bar{g}}^2+E(y)V_{\mu}^{2}(x,y)dv_{\bar{g}}\\
&\le \frac{\Lambda_\gamma(\mathbb S^n,[g_c])}{d_{\gamma}^{*}}\left(\int_{\Gamma_{\epsilon}^{0}}w_{\mu}^{2^*}dx\right)^{\frac{n-2\gamma}{2}}
+\frac{n-3}{4}\,\overline{Ric}_{yy,y}\mu^3\int_{B_{\epsilon/\mu}^{+}}y^{a+1}U_{1}^2dxdy\\
&\quad +I_3\left(-\frac{n+2}{4}\overline{Ric}_{yy,kk}+\frac{1}{2}\overline{Ric}_{yy,kk}+\bar{R}_{yiyj,ij}\right)\mu^4\\
&\quad +I_4\left(-\frac{1}{24}\overline{Ric}_{yy,yy}+\frac{\overline{Ric}_{yy,yy}}{12n}\right)\mu^4
+I_5\left(-\frac{1}{4}\overline{Ric}_{yy,kk}\right)\mu^4\\
&\quad +I_6\left(-\frac{1}{24}\overline{Ric}_{yy,yy}\right)\mu^4 +I_7 \left(\frac{n-1+a}{4}\overline{Ric}_{yy,kk}\right)\mu^4\\
&\quad +\frac{n-1+a}{12}(\overline{Ric}_{yy,yy}+2(\bar{R}_{yiyj})^2)\mu^4\int_{\R^{n+1}_{+}}y^{a+2}U^2_{1}dxdy +o(\mu^4)
\end{split}\end{equation*}
Using \eqref{2.16}, \eqref{2.17}, \eqref{2.18}  and \eqref{2.19} to simplify the coefficients, it follows that
\begin{equation}\label{formula60}\begin{split}
\int_{X}y^a&|\nabla V_{\mu}|_{\bar{g}}^2+E(y)V_{\mu}^{2}(x,y)dv_{\bar{g}}\\
\leq&\frac{\Lambda_\gamma(\mathbb S^n,[g_c])}{d_{\gamma}^{*}}
\left(\int_{\Gamma_{\epsilon}^{0}}w_{\mu}^{2^*}dx\right)^{\frac{n-2\gamma}{2}}
+\frac{n-3}{4}\,\overline{Ric}_{yy,y}\mu^3\int_{B_{\epsilon/\mu}^{+}}y^{a+1}U_{1}^2 dxdy\\
&\quad+\mu^4R_{,ii}\left\{\frac{4-n}{8(n-1)}I_3-\frac{1}{8(n-1)}I_5
+\frac{n-1+a}{8(n-1)}I_7\right\}\\
&\quad+o(\mu^4).
\end{split}\end{equation}
 Next,
using the formulas from Lemma \ref{technical}, a direct calculation shows that
\begin{equation*}\label{thetana}
(4-n)I_3-I_5+(n-1+a)I_7=\theta(n,a)\int_{\mathbb R^{n+1}_+} y^{a+2} U_1^2dxdy,
\end{equation*}
where we have defined
\begin{equation*}\begin{split}\theta(a,n):=&\frac{1}{{10n(n+2)(n-3)(3-a)(a+1)}}\cdot \left[15n^5-90n^4+(-10a^2+20a+90)n^3\right.\\
&+(20a^2-40a+300)n^2+(3a^4-12a^3+38a^2-52a-585)n\\
&+\left.(a+1)(6a^3-30a^2-114a+270)\right].
\end{split}\end{equation*}
Thanks to \eqref{2.15}, \eqref{2.17}, \eqref{2.18} and \eqref{2.19} many curvature terms vanish and
the energy \eqref{formula60} just reduces to
\begin{equation*}\begin{split}
\int_{X}y^a&|\nabla V_{\mu}|_{\bar{g}}^2+E(y)V_{\mu}^{2}(x,y)\,dv_{\bar{g}}\leq\\
&=\frac{\Lambda_\gamma(\mathbb S^n,[g_c])}{d_{\gamma}^{*}}
\left(\int_{M}w_{\mu}^{2^*}\,dv_{\hat h}\right)^{\frac{n-2\gamma}{2}}\\
&-\mu^4\frac{1}{48(n-1)}\,\theta(n,a)|W|^2(p)\int_{\R_{+}^{n+1}}y^{a+2}U_{1}^{2}dxdy+o(\mu^4).
\end{split}\end{equation*}
It is easy to show that $\theta(n,a)>0$ for $n\geq 6$. It is actually possible to show the same result for any real $n>5+2\gamma$ with the help of Matlab, but it is not relevant in our case. We may conclude that
$$\bar I_\gamma[V_\mu,\hat h]<\Lambda_\gamma(\mathbb S^n,[g_c]),$$
as desired.
Then the proof of Theorem \ref{theorem2} is completed in view of Proposition \ref{previous-work}.
\endproof

\section{Proof of Theorem \ref{theorem1}}

From the proof for Theorem 2, it is easy to see that
\begin{equation*}\begin{split}
\int_{X}y^a&|\nabla V_{\mu}|_{\bar{g}}^{2}+E(y)V_{\mu}^{2}dv_{\bar{g}}\\
&\le  \frac{\Lambda(\mathbb S^n,[g_c])}{d_{\gamma}^{*}}\left(\int_{M}w_{\mu}^{2^{*}}dv_{\hat h}\right)^{\frac{n-2\gamma}{2}}
+\frac{n-3}{4}\,\overline{Ric}_{yy,y}\mu^3\int_{B_{\epsilon/\mu}^{+}}y^{a+1}U_1^2dxdy+o(\mu^3).
\end{split}\end{equation*}
In particular, all the extra information about Einstein was used only on the terms of order $\mu^4$, so we have the same formula.

Direct calculation also shows that for $n\ge 6$,
$$\int_{R_{+}^{n+1}}y^{a+1}U_1^2(x,y)dxdy<+\infty.$$
 Indeed, we remind the reader that
$$\hat{U}=\hat{w}(|\zeta|)\phi(|\zeta|y),\quad\hat{w}(\zeta)=C_0|\zeta|^{-\gamma}K_{\gamma}(|\zeta|),\quad
\phi(s)=s^{\gamma}K_{\gamma}(s),$$
and that
$$\int_{\R^{n+1}_{+}}y^{a+1}U^2(x,y)\,dxdy
=\int_{0}^{+\infty}{s^{a+1}\phi^2(s)}\,ds\int_{\R^n}\hat{w}^2(\zeta)\frac{1}{|\zeta|^{a+2}}\,d\zeta.$$
Looking the asymptotics from Lemma \ref{bessel0}, this integral is finite when $n-4-2\gamma>-1$. Thus,
the existence of $p\in M$ such that $\overline{Ric}_{yy,y}(p)<0$ ensures the solvability of the fractional Yamabe problem, as desired. \endproof

\section{Appendix}

\begin{Lemma}
The Fourier transform of the function
$$w(x)=\lp\frac{1}{1+|x|^2}\rp^{\frac{n-2\gamma}{2}}, \quad x\in\mathbb R^n,$$
is given by
$$\hat w(\zeta)=C_0 |\zeta|^{-\gamma}K_\gamma(|\zeta|),$$
for some constant $C_0=C_0(n,\gamma)$, and $K_\gamma$ the modified Bessel function from Lemma \ref{bessel0}.
\end{Lemma}

\begin{proof}
In the following, all the equalities will be so up to multiplicative constant that may change from line to line. Since $w$ is a radial function, its Fourier transform will be radial too, and we can choose coordinate axes such that $\zeta=|\zeta|e_1$. Then, expanding in spherical coordinates,
\begin{equation*}
\hat w(\zeta)=\int_{\mathbb R^n} e^{-i x\cdot \zeta}w(x)\,dx
=\int_0^\infty \int_0^\pi e^{-i|\zeta|\cos\theta_1}(1+r^2)^{-\mu}r^{n-1}\sin^{n-2}\theta_1\, d\theta_1\,dr.
\end{equation*}
It is well known (\cite{Watson}, page 48) that
$$J_{\frac{n}{2}-1}(ar)=(ar)^{\frac{n}{2}-1}
\int_0^\pi e^{ia \cos\theta}\sin^{n-2}\theta \,d\theta,$$
and this function is real. Thus
$$\hat w(\zeta)=|\zeta|^{-\frac{n}{2}+1}\int_0^\infty r^{\frac{n}{2}} J_{\frac{n}{2}-1}(|\zeta|r)(1+r^2)^{-\mu}dr.$$
Finally, we recall  (11.4.44 in \cite{abramowitz}) that
$$\int_0^\infty \frac{r^{\nu+1} J_\nu(ar)}{(1+r^2)^\mu}\,dr=a^{\mu-1} K_{\nu-\mu+1}(a),$$
so
$$\hat w(\zeta)=|\zeta|^{-\frac{n}{2}+\mu}K_{\frac{n}{2}-\mu}(|\zeta|),$$
as desired.
\end{proof}

\end{document}